\RequirePackage{fix-cm}
\documentclass[envcountsect,numbook]{svjour3_arxiv}   
\smartqed  
\usepackage{graphicx}
\graphicspath{{./figures/}}
\usepackage[utf8]{inputenc}
\usepackage{dsfont}
\usepackage{subfigure}
\usepackage{color}
\usepackage{url}
\usepackage{enumerate} 
\usepackage{booktabs}
\usepackage{tabularx}
\usepackage{ltablex}
\usepackage{mathptmx}
\usepackage{amssymb,amsmath,amsfonts,latexsym} 
\usepackage{longtable} 
\usepackage{textcomp}
\usepackage{lscape}
\usepackage[12pt]{extsizes}
\usepackage{geometry}
\usepackage{ifthen}
\usepackage{mathtools}
\mathtoolsset{showonlyrefs}
\usepackage{calrsfs}
\DeclareMathAlphabet{\pazocal}{OMS}{zplm}{m}{n}
\let\mathcal\pazocal
\allowdisplaybreaks

\spnewtheorem{theorem}{Theorem}{\bfseries}{\itshape}
\spnewtheorem{corollary}[theorem]{Corollary}{\bfseries}{\itshape}
\spnewtheorem{lemma}[theorem]{Lemma}{\bfseries}{\itshape}
\spnewtheorem{proposition}[theorem]{Proposition}{\bfseries}{\itshape}
\spnewtheorem{definition}[theorem]{Definition}{\bfseries}{\itshape}
\spnewtheorem{remark}[theorem]{Remark}{\bfseries}{\upshape}
\spnewtheorem{assumption}[theorem]{Assumption}{\bfseries}{\itshape}
\counterwithin{theorem}{section}
\smartqed

\renewcommand{\paragraph}[1]{{\bf #1.}}
\definecolor{myred}{rgb}{0.8,0,0}  

{\vskip\baselineskip\noindent\textbf{Proof of {#1}:}}%
{\hspace*{.1pt}\hspace*{\fill}\BOX\vskip\baselineskip}

\usepackage{tikz}

\def \R{\mathbb{R}}               
\def \N{\mathbb{N}}               
\def \C{\mathbb{C}}               
\def \1{{\bf 1}}                
\def \0{{\bf 0}}
\def\qed{\hfill$\Box$}

\def \diag{\operatorname{diag}}

\definecolor{myred}{rgb}{0.9,0,0}  
\definecolor{mygreen}{rgb}{0,0.7,0}  
\definecolor{myblue}{rgb}{0.2,0,0.8}  
\definecolor{orange}{rgb}{1,0.6,0}  
\definecolor{olive}{rgb}{0.5,0.5,0}  
\definecolor{mylila}{rgb}{0.8,0.5,0.2}  
\definecolor{mygrey}{rgb}{0.6,0.6,0.6}  
\definecolor{mybrown}{rgb}{0.65,0.16,0.16}  
\definecolor{mymaroon}{rgb}{0.11,0.0,0.0}

\def \Domainspace{\mathcal{D}}  

\def \phx{\text{PHX} }
\def \phxk{\text{PHX}}
\def \phxs{\text{PHX}s }
\def \phxsk{\text{PHX}s}

\def \medium{M}
\def \fluid{F}
\def \outlet{O}
\def \inlet{I}
\def \bottom{B}
\def \interface{J}
\def \storage{S}

\def \Qin{Q^{\inlet}}
\def \QinC{\Qin_C}
\def \QinD{\Qin_D}

\def \Qg{Q^{G}}
\def \Qmm{Q^\medium}
\def \Qff{Q^\fluid}
\def \Dm{\Domainspace^\medium}
\def \Df{\Domainspace^\fluid}
\def \Dmf{\Domainspace^{\storage}}
\def \Din{\Domainspace^{I}}
\def \Dout{\Domainspace^{\outlet}}
\def \Dbottom{\Domainspace^\bottom}
\def \Dtop{\Domainspace^{T}}
\def \Dleft{\Domainspace^{L}}	
\def \Dright{\Domainspace^{R}}
\def \DInterface{\Domainspace^{\interface}}		
\def \DInterfaceL{\underline{\Domainspace}^{\interface}}
\def \DInterfaceU{{\overline{\Domainspace}}{}^{\interface}}

\def \rhom{\rho^\medium}
\def \rhof{\rho^\fluid}
\def \kappam{\kappa^\medium}
\def \kappaf{\kappa^\fluid}
\def \cp{c_p}
\def \cpm{\cp^\medium}
\def \cpf{\cp^\fluid}
\def \am{a^\medium}
\def \af{a^\fluid}

\def \gamam{\gamma^\medium}
\def \gamaf{\gamma^\fluid} 
\def \weightf{\weight^\fluid} 
\def \weightm{\weight^\medium} 
\def \alpham{\alpha^\medium}

\def \betaFm{\beta^{\fluid/\medium}}
\def \betam{\beta^{\medium}}
\def \betaf{\beta^{\fluid}}
\def \Nfm{\mathcal{N}^{\fluid \medium}}
\def \Nmed{\mathcal{N}^\medium}
\def \Nfluid{\mathcal{N}^\fluid}
\def \Ninter{\mathcal{N}^\interface}
\def \NinterL{\underline{\mathcal{N}}^\interface}
\def \NinterU{\overline{\mathcal{N}}^\interface}
\def \NinletB{\mathcal{N}_\inlet^{\mathcal{B}}}
\def \NoutB{\mathcal{N}_\outlet^{\mathcal{B}}}
\def \NbottomB{\mathcal{N}_\bottom^{\mathcal{B}}}
\def \NtopB{\mathcal{N}_T^{\mathcal{B}}}
\def \NleftB{\mathcal{N}_L^{\mathcal{B}}}
\def \NrightB{\mathcal{N}_R^{\mathcal{B}}}
\def \betam {\beta^\medium}
\def \betaf{\beta^\fluid}
\newcommand{\mycaption}[1]{\caption{\footnotesize #1}}
\newcommand{\ncol}{q}
\newcommand{\vconst}{\overline{v}_0}
\newcommand{\weight}{\psi}
\newcommand{\heattransfer}{\lambda^{\!G}}

\newcommand{\mat}[1]{{#1}}
\newcommand{\matgenel}{M}
\newcommand{\matgen}{\mat{\matgenel}}
\newcommand{\normalvec}{\mathfrak{n}}   

\newcommand{\dom}{\dagger}

\begin{document}
	
	\title{Short-Term Behavior of a Geothermal Energy Storage: Modeling and Theoretical Results  
}

\titlerunning{Short-Term Behavior of a Geothermal Energy Storage:  Modeling and Theoretical Results}        

\author{Paul Honore Takam  \and Ralf Wunderlich \and Olivier Menoukeu Pamen}

\authorrunning{P.H.~Takam, R.~Wunderlich, O.~Menoukeu Pamen} 

\institute{Paul Honore Takam \at
	Brandenburg University of Technology Cottbus-Senftenberg, Institute of Mathematics, P.O. Box 101344, 03013 Cottbus, Germany;  
	\email{\texttt{takam@b-tu.de}}           
	\and
	Ralf Wunderlich \at
	Brandenburg University of Technology Cottbus-Senftenberg, Institute of Mathematics, P.O. Box 101344, 03013 Cottbus, Germany;  
	\email{\texttt{ralf.wunderlich@b-tu.de}} 
	\and
	Olivier Menoukeu Pamen \at 
	University of Liverpool, Department of Mathematical Sciences, Liverpool L69 3BX, United Kingdom; 
	\email{\texttt{O.Menoukeu-Pamen@liverpool.ac.uk}} 
}

\date{Version of  \today}

\maketitle

\begin{abstract}
	This paper investigates numerical methods for simulations of  the short-term behavior of a geothermal energy storage. Such simulations are needed for the optimal control and management  of residential heating systems equipped with an underground  thermal storage. There a given volume under or aside of  a building is  filled with soil and insulated to the surrounding ground. The thermal energy is stored by raising the temperature of the soil inside the storage. It is charged and discharged via pipe heat exchangers filled with a moving fluid. Simulations of geothermal energy storages aim to determine how much energy  can be stored in or taken from the storage within a given  short period of time. The latter depends on the dynamics of the spatial temperature distribution in the storage	which is governed by a linear heat equation with convection and appropriate boundary and interface conditions.  We consider semi- and full discretization of that PDE  using  finite difference schemes and study associated stability problems. Numerical results  based on the derived methods are presented in the companion paper \cite{Takam2021NumResults}.
	
	%
\end{abstract}	
\section{Introduction}
Thermal storage facilities help to mitigate and to manage temporal fluctuations of  heat supply and demand for heating and cooling systems of single buildings as well as for district heating systems. They  allow heat to be stored in form of thermal energy and be used hours, days, weeks or months later. This is attractive  for space heating, domestic or process hot water production, or generating electricity. Note that  thermal energy may also be stored in the way of cold.  Thermal storages can significantly increase both the flexibility and the performance of district energy systems and enhancing the integration of intermittent  renewable energy 	sources into thermal networks (see Guelpa and  Verda \cite{guelpa2019thermal}, Kitapbayev et al.~\cite{KITAPBAYEV2015823}).	   Since heat production is still  mainly based on burning fossil fuels (gas, oil, coal) these are important contributions for the reduction of carbon emissions and an increasing energy independence of societies.

For an overview on thermal energy storages we refer to  Dincer and Rosen \cite{dincer2021thermal}. Zalba et al.~\cite{zalba2003review} provides a review  for the history of thermal energy storages with solid–liquid phase change and focused in three aspects: materials, heat transfer and applications. An overview of the  European and in particular the Spanish thermal energy storage potential
is presented in  Arce et al. \cite{arce2011overview}. The authors show that  thermal energy storages make an important contribution to the  reduction of  CO$_2$-emissions.  	 
In  Soltani et al. \cite{soltani2019comprehensive} the authors provide a comprehensive review on the evolution of geothermal energy production from its  beginnings to the present time by reporting production data from individual countries and
collective data of worldwide production. 		

The efficient operation of  geothermal storages requires a thorough design and planning because of the considerable  investment cost. For that purpose,  mathematical models and  numerical  simulations  are widely used. We refer to Dahash et al. \cite{dahash2020toward} and the references therein. In that paper the authors investigate  large-scale seasonal thermal energy storages allowing for buffering intermittent renewable heat production in district heating systems.  Numerical simulations are based on a multi-physics model of the thermal energy storage which was calibrated to measured data for a pit thermal energy storage in Dronninglund (Denmark). 
Another contribution is  Major et al. \cite{major2018numerical} which considers heat storage capabilities of deep sedimentary reservoirs. The governing heat and flow equations are solved using finite element methods.  Further, Regnier et al.~\cite{Regnier_et_al_2022} study the numerical simulation of aquifer thermal energy storages and focus on dynamic mesh optimisation for the finite element solution of the heat and flow equations.   For further contributions to the numerical simulation of such storages we refer  to \cite{bazri2022thermal,dincer2021thermal,haq2016simulated,li2022modelling,soltani2019comprehensive,wu2022enhancing}.

In this paper  we focus on  geothermal storages as depicted in  Fig.~\ref{fig:etank}.
\begin{figure}[!h]
	\centering
	\includegraphics[width=0.9\textwidth]{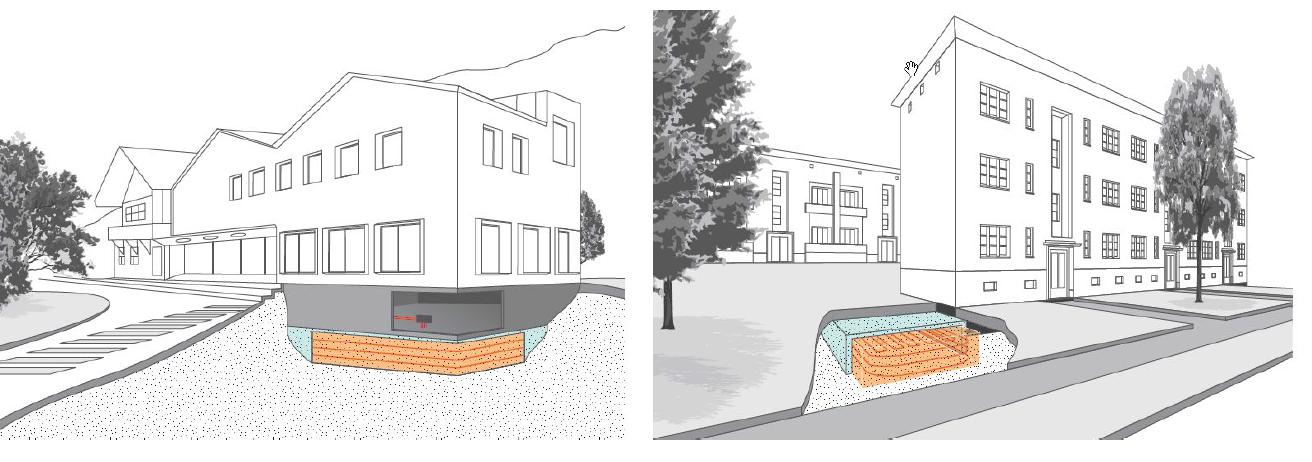}
	\caption[Geothermal storage]{Geothermal storage: in the new building, under a building (left) and in the renovation, aside of  the building (right), see
		\url{www.ezeit-ingenieure.eu}, \url{www.geo-ec.de}.}
	\label{fig:etank}
\end{figure}
Such storages gain more and more importance and are quite attractive for residential heating systems since  construction and maintenance  are relatively inexpensive. Furthermore, they can be integrated both in new buildings and in renovations. 
We will work with  a 2D-model of   a geothermal thermal energy storage, see Fig.~\ref{etank_longtermsimu},  where a defined volume under or aside of a building is filled with soil and insulated to the surrounding ground.  Thermal energy is stored by raising the temperature of the soil inside the storage. It is charged and discharged via pipe heat exchangers  (\phxk) filled with some fluid (e.g.~water). These \phxs can be connected to a short-term storage such as a water tank or directly to a solar collector and (heat) pumps move the fluid carrying the thermal energy.
A special feature of the storage in this work is that it is not insulated at the bottom such that thermal energy can also flow into deeper layers {  as it can be seen in Fig.~\ref{etank_longtermsimu}.} This can be considered as a natural extension of the storage capacity since that heat can to some extent be retrieved if the storage is sufficiently discharged  (cooled) and a heat flux back to storage is induced. Of course,  there are unavoidable diffusive losses  to the  environment  but  due to the open architecture, the geothermal storage can benefit from higher temperatures in deeper layers of the ground  and serve as a production unit  similar to a downhole heat exchanger. Note that in many regions in Europe the temperature in a depth of only 10 meter   is  almost constant around 10${}^\circ C$ over the year.	

\begin{figure}[h!]
	\centering
	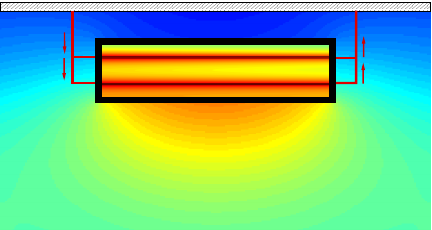
	\mycaption{\label{etank_longtermsimu} 
		2D-model of a geothermal storage  insulated to the top and the sides  while open at the bottom and spatial temperature distribution.}
\end{figure}

Geothermal storages enable an extremely efficient operation of heating and cooling systems in buildings. Further, they can be used  to mitigate peaks in the electricity grid by converting electrical into heat energy (power to heat). Pooling several geothermal storages within the  framework of a virtual power plant gives the necessary  capacity which allows to participate in the balancing energy market.

This paper extends and complements the results in B\"ahr et al.~\cite{bahr2017fast,bahr2022efficient} where the authors focus  on the numerical simulation of the long-term behavior over weeks and months of the spatial temperature distribution in a geothermal storage and the interaction between a geothermal storage and its surrounding domain.  For simplicity charging and discharging was described by a simple source term  but not by \phxsk.

In the present work we  focus on the computation of the short-term behavior of the spatial temperature distribution. This is needed for storages embedded into residential heating systems and the study of the storage's  response to charging and discharging operations on time scales from a few minutes to a few days. We extend the setting in \cite{bahr2017fast,bahr2022efficient} and include \phxs  for a more realistic model of the storage's  charging and discharging process.  However,  for the sake of simplicity we do not consider the surrounding medium but  reduce the computational domain to  the storage  depicted in Fig.~\ref{etank_longtermsimu} by a  black rectangle. Instead we 
set appropriate boundary conditions to mimic the interaction between storage and  environment.   

For the management and control of a storage which is embedded into a residential heating system 	one needs to know 	 the amount of available thermal energy that can be stored in or extracted from the storage in a given short period of time. Such questions can only be answered if one knows the spatial temperature distribution,  in particular around the \phxsk. Charging and discharging is not efficient or even impossible if there are only small differences between the temperatures  inside and in the vicinity of the \phxsk. Long periods of (dis)charging may lead to saturation  in the vicinity of the \phxsk. As a consequence (dis)charging is no longer efficient and should be stopped since propagation of heat to regions away from the \phxs takes time.	

The  short-term behavior of  the spatial temperature distribution is governed by a linear heat equation with convection and appropriate boundary and interface conditions. We solve that PDE  using finite difference schemes, see Duffy \cite{duffy2013finite}.  For the convection terms we	apply  upwind techniques.
In a first step we study the semi-discretization with respect to spatial variables leading to a system of linear ODEs. In a second step, we consider full space-time discretization and   derive  implicit finite-difference schemes. The  current  paper provides the following theoretical  contributions. First, we prove that the chosen semi-discretization ensures a system of linear ODEs with a stable system matrix. Second, we provide a detailed stability analysis for the implicit finite-difference schemes of the fully discretized PDE and establish a stability condition. 

Numerical results  are devoted to  our  companion  paper  \cite{Takam2021NumResults}.
There we perform extensive numerical experiments, where simulations results for the temporal behavior of the spatial temperature distribution are used  to determine how much energy  can be stored in or taken from the storage within a given  short period of time. Special focus is laid on the dependence of these quantities on the arrangement of the \phxs within the storage.	
Further, we refer to another companion paper \cite{Takam2020Reduction} in which we apply model reduction techniques known from control theory such as balanced truncation to derive low-dimensional approximations of aggregated characteristics of the temporal behavior of the spatial temperature distribution. The latter is crucial if the geothermal storage is embedded into a  residential heating system 
and the cost-optimal management of such systems is studied mathematically in terms of optimal control problems.

The rest of the paper is organized as follows. In Sec.~\ref{ExternalStorage} we describe the dynamics of the spatial temperature distribution in the geothermal storage which is governed by a linear heat equation with a convection term and appropriate boundary and interface conditions. In Sec.~\ref{Discretization} we present the semi-discretization with respect to spatial variables  of the initial boundary value problem for that heat equation. 
For the resulting system of linear ODEs we show that the system matrix is stable. The full space-time discretization is studied in Sec.~\ref{sec:full_discrtization} where we derive implicit finite-difference schemes and provide the associated stability analysis.		 
An appendix provides a list of frequently used notations,  some technical details  of the finite difference scheme, auxiliary results from matrix analysis  as well as proofs 		 which were removed from the main text.

\section{ Dynamics of Spatial Temperature Distribution in a  Geothermal Storage}
\label{ExternalStorage}
In this section we describe  the dynamics of the  spatial temperature distribution in a geothermal storage mathematically by a linear heat equation with  convection term and appropriate boundary and interface conditions.

\subsection{2D-Model}
\label{model2D}

We assume that the domain of the geothermal storage is a cuboid and consider a two-dimensional rectangular cross-section.
We denote by  $Q=Q(t,x,y)$ the temperature at time $t \in [0,T]$ in the point $(x,y)\in \Domainspace=(0,l_x) \times (0,l_y)$ with $l_x,l_y$ denoting the width and height of the storage. 
The domain $\Domainspace$ and its boundary $\partial \Domainspace$  are depicted  in Fig.~\ref{bound_cond}.  $\Domainspace$ is divided into three parts. The first is $\Dm$ and is  filled with a homogeneous medium (soil) characterized by constant  material parameters $\rhom, \kappam$ and $\cpm$ denoting  mass density,    thermal conductivity and   specific heat capacity, respectively. The second is $\Df$,  it represents the \phxs filled with a fluid (water) with constant material parameters $\rhof, \kappaf$ and $\cpf$. The fluid moves with time-dependent velocity $v_0(t)$ along the \phxk. For the sake of simplicity we  restrict ourselves to the case, often observed in applications, where the pumps moving the fluid are either on or off. Thus the velocity $v_0(t)$ is piecewise constant taking  values $\vconst>0$ and zero, only.  Finally, the third part is the interface $\DInterface$ between $\Dm$ and $\Df$. That interface is split into  upper and lower interfaces $\DInterfaceU$ and $\DInterfaceL$, respectively. Observe that we neglect modeling the wall of the \phx and suppose perfect contact between the \phx and the soil. Details are given  in \eqref{Interface} and \eqref{eq: 13f} below. Summarizing we make the following 
\begin{figure}[h!]			
	\begin{center}
		\includegraphics[width=0.8\linewidth,height=.6\linewidth]{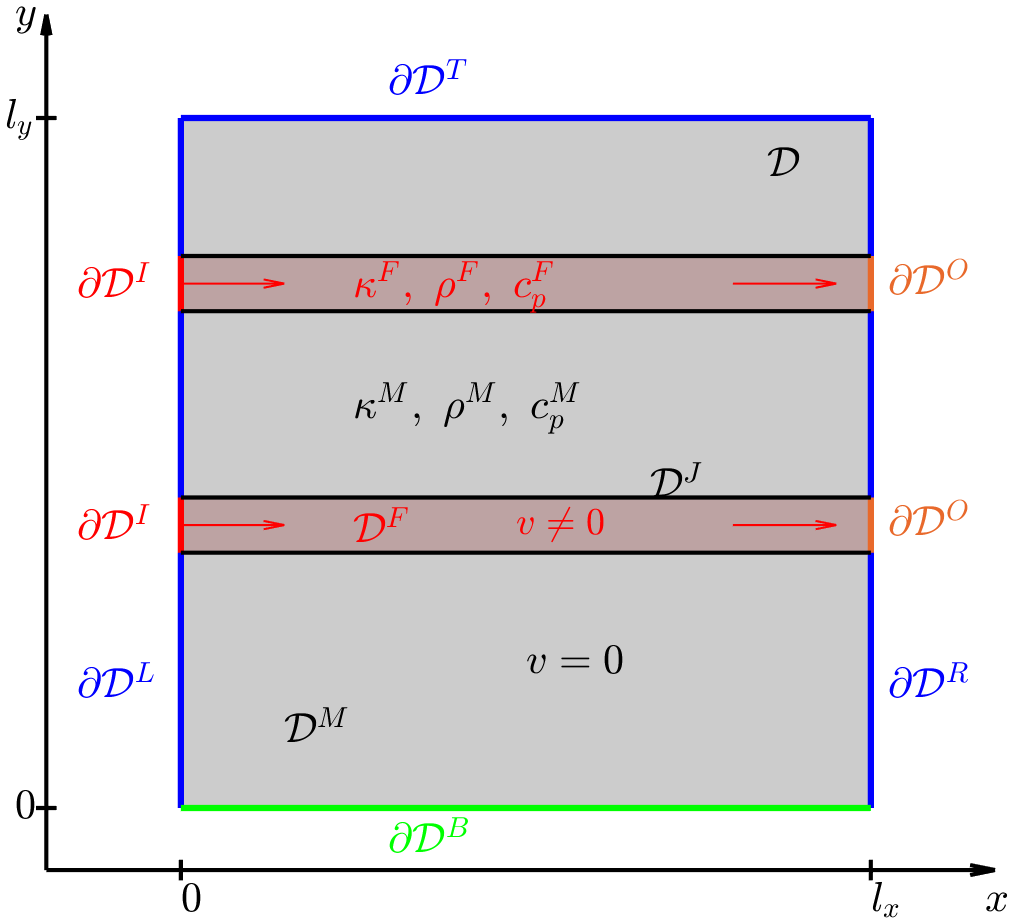}
	\end{center}
	\mycaption{\label{bound_cond} 2D-model of the geothermal storage:  decomposition of the domain $\Domainspace$ and the boundary $\partial \Domainspace $.}			
\end{figure}
\begin{assumption}
	\label{assum1}~
	\begin{itemize}
		\item [1.] Material parameters of the medium   $\rhom, \kappam, \cpm$ in the domain $\Dm$  and of the fluid  $\rhof, \kappaf, \cpf$ in the domain  $\Df$ are constants.
		\item [2.] Fluid velocity is piecewise constant, i.e. $v_0(t)=\begin{cases}
			\vconst>0,  &\text{pump~on},\\
			0, & \text{pump off}.
		\end{cases}$
		\item [3.] Perfect contact at the interface between fluid and  medium.
	\end{itemize}
\end{assumption}

\paragraph{Heat equation}
The temperature $Q=Q(t,x,y)$ in the external storage is governed by the linear heat equation with convection term
\begin{align}
	\rho c_p \frac{\partial Q}{\partial t}=\nabla \cdot (\kappa \nabla Q)-
	\rho v \cdot \nabla (c_p Q),\quad  (t,x,y) \in (0,T]\times \Domainspace \backslash \DInterface,   \label{heat_eq}
\end{align}
where $\nabla=\big(\frac{\partial}{\partial x},\frac{\partial}{\partial y}\big)$ denotes the gradient operator. The first term  on the right hand side describes diffusion while the second represents convection of the moving fluid in the \phxsk.
Further,  
$v=v(t,x,y)$ $=v_0(t)(v^x(x,y),v^y(x,y))^{\top}$  denotes  the velocity vector with $(v^x,v^y)^\top$ being the normalized directional vector of the flow. 
According to Assumption \ref{assum1} the material parameters $\rho,\kappa, c_p$ 
depend on the position $(x,y)$ and take the values $\rhom,\kappam, \cpm$ for points in $\Dm$ (medium) and  $\rhof, \kappaf, \cpf$ in  $\Df$ (fluid).

Note that there are no sources or sinks inside the storage and therefore the above heat equation appears without forcing term.
Based on this assumption, the heat equation (\ref{heat_eq}) can be written as
\begin{align}
	\frac{\partial Q}{\partial t}=a\Delta Q-
	v \cdot \nabla Q,\quad  (t,x,y) \in (0,T]\times \Domainspace \backslash \DInterface,   \label{heat_eq2}
\end{align}
where $\Delta=\frac{\partial^2}{\partial x^2}+\frac{\partial^2}{\partial y^2}$ is the Laplace operator and  $ a=a(x,y)$ is the thermal diffusivity which is piecewise constant with values  $a^\dom=\frac{\kappa^\dom}{\rho^\dom \cp^\dom}$  with $\dom=\medium$ for  $(x,y)\in \Dm $ and $\dom=\fluid$   for  $(x,y)\in \Df$, respectively.  
The initial condition $Q(0,x,y)=Q_0(x,y)$ is given by the initial temperature distribution $Q_0$ of the storage.

\begin{remark}\label{rem_heat_exchanger}
	In real-world geothermal storages \phxs  are often designed in a snake form located in the storage domain at multiple horizontal layers. There may be  only a single inlet and a single outlet. We will mimic that design by a computationally more tractable design characterized by multiple horizontal straight \phxs as it is sketched in Fig.~\ref{etank_longtermsimu}.
	This allows to control   the \phxs in different layers separately. For a  topology with single inlet and  outlet snake-shaped \phxs  the outlet of a straight \phx in one layer  can be connected with the inlet of the straight \phx in the next layer. 
\end{remark}

\subsection{Boundary and Interface Conditions}
For the description of the boundary conditions we decompose the boundary $\partial\Domainspace$ into several subsets as depicted in Fig.~\ref{bound_cond} representing the insulation on the top and the side, the open bottom, the inlet and outlet of the \phxsk.  Further, we have to specify conditions at the interface between  \phxs and  soil. The inlet, outlet and the interface  conditions model the heating and cooling of the storage via \phxsk. We distinguish between the two regimes 'pump on' and 'pump off'. For simplicity we assume perfect insulation at inlet and outlet if the pump is off.
This leads to the following boundary conditions.

\begin{itemize}				
	\item \textit{Homogeneous Neumann condition} describing perfect insulation on the top  and the side 
	\begin{align}\frac{\partial Q}{\partial \normalvec }=0,\qquad (x,y)\in 
		\partial \Dtop \cup \partial \Dleft  \cup \partial \Dright , 
		\label{Neumann}
	\end{align}
	where $\partial \Dleft =\{0\} \times  [0,l_y] \backslash \partial \Din $, ~
	$\partial \Dright =\{l_x\} \times  [0,l_y] \backslash \partial \Dout$, $\partial \Dtop = [0,l_x] \times \{l_y\}$   and $\normalvec$ denotes the outer-pointing normal vector.
	\item \textit{Robin condition} describing heat transfer at the bottom 
	\begin{align}
		-\kappam\,\frac{\partial Q}{\partial \normalvec }=\heattransfer(Q-\Qg(t)), \qquad  (x,y)\in 
		\partial \Dbottom ,
		\label{Robin}
	\end{align}
	with $\partial \Dbottom = [0,l_x] \times \{0\}$, where $\heattransfer>0$ denotes the  heat transfer coefficient  and $\Qg(t)$ the underground temperature.
	For more interpretation we refer to Remark \ref{rem:Robin}.
	\item  \textit{Dirichlet condition} at the inlet if the pump is on ($v_0(t)>0$), i.e.~the fluid arrives at the storage with a given temperature $\Qin(t)$. If pump is off ($v_0(t)=0$), we set a homogeneous Neumann condition describing perfect insulation.
	\begin{align}
		\begin{cases}
			\begin{array}{rll}	
				Q &=\Qin(t), &\text{ ~pump on,} \\
				\frac{\partial Q}{\partial \normalvec }&=0, &\text{ ~pump off,} 
			\end{array}
		\end{cases} 
		\qquad  (x,y)\in 	\partial \Din .
		\label{input}
	\end{align}
	
	\item \textit{``Do Nothing'' condition} at the outlet in the following sense. If  the  pump is on ($v_0(t)>0$) then the total heat flux directed outwards can be decomposed into a diffusive  heat flux given by $\kappaf\frac{\partial Q}{\partial \normalvec }$ and a convective  heat flux given by $v_0(t) \rhof \cpf Q$. Since in real-world applications the latter is much larger than the first we neglect the diffusive heat flux. This leads to a homogeneous Neumann condition
	\begin{align}\frac{\partial Q}{\partial \normalvec }=0,\qquad (x,y)\in 
		\partial \Dout . 
		\label{output}
	\end{align}
	If the pump is off then we  assume (as already for the inlet) perfect insulation which is also described by the above condition.
	
	\item \textit{Smooth heat flux} at interface $\DInterface$ between fluid and soil leading to a coupling condition
	\begin{align}
		\kappaf\,\frac{\partial \Qff }{\partial \normalvec  }=\kappam\,\frac{\partial \Qmm }{\partial \normalvec  },
		\qquad  (x,y)\in \DInterface.
		\label{Interface}
	\end{align}	
	Here, $\Qff , \Qmm $ denote the temperature of the fluid inside the \phx and of the soil outside the \phxk, respectively.
	Moreover, we assume that the contact between the \phx and the medium is perfect which leads to a smooth transition of a temperature, i.e., we have 
	\begin{align}
		\Qff =\Qmm ,\qquad  (x,y)\in \DInterface. \label{eq: 13f}
	\end{align} 
\end{itemize}
\begin{remark}
	If the contact between the \phx and the medium is not perfect (e.g., in case of contact resistance) then the transition of the temperature at the interface $\DInterface$ will not be smooth, that is, $\Qff  \neq \Qmm .$
	This leads to a temperature jump between the \phx and the medium. That phenomenon occurs in the heat transfer between the medium and an insulation as shown in \cite{bahr2017fast,bahr2022efficient}. 
\end{remark}

\begin{remark}\label{rem:Robin}
	Imposing the Robin condition \eqref{Robin} at the bottom boundary aims to  mimic the thermal behavior at the bottom boundary. A more realistic description requires embedding the storage domain $\Domainspace$ into a larger computational domain including the  surrounding  regions as  in Fig.~\ref{etank_longtermsimu}. This allows for warming and cooling in the vicinity of the storage resulting from  the outflow and inflow of the storage heat.  Contrary to that, condition \eqref{Robin} assumes an exogenously given underground temperature $\Qg$ independent of the temperature in the storage.
	
	The  heat transfer coefficient $\heattransfer$ describes the  resistance to the heat flux at the boundary. For the limiting case $\heattransfer\to 0$ we get a homogeneous Neumann condition, i.e., perfect insulation, while in the limit for $\heattransfer\to \infty$ condition \eqref{Robin} is the Dirichlet condition $Q=\Qg(t)$.
	The underground temperature in general shows seasonal fluctuations which can be described by
	$
	\Qg(t)=K^G_1 \cos\big(\frac{2\pi (t-t_0)}{T_{a}}\big)+K^G_2, 
	$ 
	where $K^G_1$ is the intensity of the fluctuation, $K^G_2$ is the average ground temperature,  $t_0$ a time or phase shift and $T_a$ the number of time units per year. Since our focus is on the short-term behavior, we assume in the sequel that the underground temperature is constant over time, i.e. $K^G_1 =0$. 
\end{remark}	

\section{Semi-Discretization of the Heat Equation}
\label{Discretization}
This and the next section are devoted to the finite difference discretization of  the heat equation \eqref{heat_eq2} with the boundary and interface conditions \eqref{Neumann} through \eqref{eq: 13f}.  
We proceed in two steps. In the first step we apply semi-discretization in space and approximate only spatial derivatives by their respective finite differences.  This approach is also known as 'method of lines' and leads  to a high-dimensional system of ODEs with a stable system matrix for the temperatures at the grid points.  The latter will be used as starting point for model reduction  in our paper \cite{Takam2020Reduction}. In the second step, see Sec.~\ref{sec:full_discrtization}, also time is discretized  resulting in  a family of implicit finite difference schemes  for which we perform a stability analysis.

\subsection{Semi-Discretization of the Heat Equation}
\label{Semi-Discret}
We now apply the finite difference method (see Duffy \cite{duffy2013finite}) combined with upwind techniques for the convection terms
for semi-discretization  of the heat equation (\ref{heat_eq2}). 

\begin{figure}[h!]
	\centering
	\begin{center}
		\resizebox{0.7\textwidth}{0.5\textwidth}{%
			\begin{tikzpicture}[thick,scale=1, every node/.style={scale=1}]
				\draw[thick,->] (-1.5,-1.5) -- (8.7,-1.5) node[anchor=north west] {x};
				\draw[thick,->] (-1.5,-1.5) -- (-1.5,8.3) node[anchor=south east] {y};
				\draw[step=1.5cm,black,very thin] (-1.5,-1.5) grid (7.5,7.5);					
				\fill (3,3) circle (3pt);
				\node at (3.5,2.7) {$(i,j)$};
				\fill (3,4.5) circle (3pt);
				\node at (3.8,4.8) {$(i,j+1)$};
				\fill (3,1.5) circle (3pt);
				\node at (3.8,1.2) {$(i,j-1)$};
				\fill (4.5,3) circle (3pt);
				\node at (5.3,2.7) {$(i+1,j)$};
				\fill (1.5,3) circle (3pt);
				\node at (2.2,2.7) {$(i-1,j)$};
				\node at (7.7,-2) {$l_x=N_xh_x$};
				\node at (-2.6,7.5) {$l_y=N_yh_y$};
				\fill (7.5,7.5) circle (3pt);
				\node at (8.3,7.75) {$(N_x,N_y)$};
				\fill (-1.5,-1.5) circle (3pt);
				\node at (-2.1,-1.8) {$(0,0)$};	
				\fill (-1.5,7.5) circle (3pt);
				\node at (-0.8, 7.75) {$(0,N_y)$};
				\fill (7.5,-1.5) circle (3pt);
				\node at (8.2,-1.1) {$(N_x,0)$};
			\end{tikzpicture}
		} 
	\end{center}
	\mycaption{\label{grid}Computational grid.}	
\end{figure}
Let $N_x$ and $N_y$ be the number of grid points  and  $h_x={l_x}/{N_x}$ and $h_y={l_y}/{N_y}$  the step sizes in $x$-direction and $y$-direction, respectively.  
The spatial domain is discretized by means of a mesh   with grid points  $(x_i,y_j)$ as shown in Fig.~\ref{grid} where
\begin{align*}
	x_i =ih_x, ~~~y_j =jh_y,\quad i ={ 0},...,N_x, ~~~j ={ 0},...,N_y.
\end{align*}

We denote by $Q_{ij}(t)\simeq Q(t,x_i,y_j)$ the semi-discrete approximation  of the temperature   and by $v_0(t)(v^x_{ij},v^y_{ ij})^{\top}=v_0(t)(v^x(x_i,y_j),v^y(x_i,y_j))^{\top}=v(t,x_i,y_j)$  the velocity vector  at the grid point $(x_i,y_j)$ at time $t$. Further, we introduce the following sets of indices
\begin{align*}
	\mathcal{N}_x&=\{1,...,N_x-1\}, ~~~
	\mathcal{N}_y=\{1,...,N_y-1\},\\
	\Nmed &=\{(i,j): ~(i,j) \in \mathcal{N}_x \times \mathcal{N}_y   ~~\text{with} ~~(x_i,y_j) \in \Dm \},\\
	\Nfluid &=\{(i,j):~ (i,j) \in \mathcal{N}_x \times \mathcal{N}_y   ~~\text{with} ~~(x_i,y_j) \in \Df \},\\
	\Ninter &=\{(i,j):~ (i,j) \in \mathcal{N}_x \times \mathcal{N}_y   ~~\text{with} ~~(x_i,y_j) \in \DInterface \},\\
	\mathcal{N^{\mathcal{B}}}&=\{(i,j):~ (i,j) \in \{0,...,N_x\} \times \{0,...,N_y\}   ~~\text{with} ~~(x_i,y_j) \in \mathcal{\partial D} \},
\end{align*}
which we identify with the corresponding sets of grid points.
We denote by $\mathcal{N}^{\storage} =\Nfluid \cup \Nmed $ the set of grid points in the inner domain $\Dmf=\Df \cup \Dm$. Further, we decompose  the set of grid points  on the interface $\DInterface=\DInterfaceL \cup \DInterfaceU$
between the fluid and medium into $\Ninter =\NinterL \cup \NinterU$.  Here,  $\DInterfaceL$ and $\DInterfaceU$ denote the lower and upper interface, respectively, see Fig.~\ref{bound_cond}. Further, we decompose the set $\mathcal{N^{\mathcal{B}}}$ of grid points on the boundary domain $\partial \Domainspace$  according to the decomposition of $\partial \Domainspace$ given in Fig.~\ref{bound_cond} into  $\mathcal{N}^{\mathcal{B}}=\NinletB \cup \NoutB \cup \NleftB \cup \NrightB \cup \NtopB \cup \NbottomB $.

The spatial derivatives in the PDE (\ref{heat_eq}) are approximated
by linear combinations of values of $Q$ at the grid points $(x_i,y_j)$  in $\Dmf$ at time $t$. 
We use  central second-order  finite difference for the diffusion term: 
\begin{align*}
	\frac{\partial^2 Q(t,x_i,y_j)}{\partial x^2} = \frac{Q_{i+1,j}(t)-2Q_{ ij}(t)+Q_{i-1,j}(t)}{h^2_x}+\mathcal{O}(h^2_x), \\
	\frac{\partial^2 Q(t,x_i,y_j)}{\partial y^2} = \frac{Q_{i,j+1}(t)-2Q_{ ij}(t)+Q_{i,j-1}(t)}{h^2_y}+\mathcal{O}(h^2_y). 
\end{align*}	
For the convection term we use the upwind discretization to get 	
\begin{align*}
	v^x(x_i,y_j)\frac{\partial Q(t,x_i,y_j)}{\partial x} &= v_{ ij}^x \mathds{1}_{\{v_{ ij}^x>0\}} \frac{Q_{ ij}(t)-Q_{i-1,j}(t)}{h_x}  \nonumber \\& 
	+ v_{ij}^x \mathds{1}_{\{v_{ ij}^x<0\}} \frac{Q_{i+1,j}(t)-Q_{ ij}(t)}{h_x}+\mathcal{O}(h_x), 
	\\
	v^y(x_i,y_j)\frac{\partial Q(t,x_i,y_j)}{\partial y} &= v_{ ij}^y \mathds{1}_{\{v_{ ij}^y>0\}} \frac{Q_{ ij}(t)-Q_{i,j-1}(t)}{h_y} \nonumber \\&
	+v_{ ij}^y \mathds{1}_{\{v_{ ij}^y<0\}} \frac{Q_{i,j+1}(t)-Q_{ ij}(t)}{h_y}+\mathcal{O}(h_y). 
\end{align*}
We have to point out that theabove upwind approximations 
of the convection terms need to be  applied only to the set of grid points  $\Nfluid $ in the fluid domain $\Df $, since there is no convection outside the fluid  and we can set $v_{ij}^x=v_{ij}^y=0$.

For the sake of simplification and tractability of our analysis we restrict ourselves  to the following assumption on the arrangement of \phxs and impose  conditions on the location of grid points  along the \phxsk.
\begin{assumption}\label{assum2}~%
	\begin{enumerate}
		\item  There are $n_P \in \N$  straight horizontal \phxsk, the fluid moves in positive $x$-direction.
		\item The interior of \phxs contains grid points.
		\item Each interface between medium and fluid contains grid points.
	\end{enumerate}		
\end{assumption}	
Then for grid points in the domain
$\Dmf $  the semi-discrete scheme is given by

\begin{align}
	\frac{dQ_{ij}(t)}{dt} &~= \alpha_{ij}^{+}(t) Q_{i+1,j}(t)+\alpha_{ij}^{-}(t) Q_{i-1,j}(t)+\beta_{ij}^{+}(t) Q_{i,j+1}(t) +\beta_{ij}^{-}(t) Q_{i,j-1}(t)\nonumber\\&~~~~+\gamma_{ij}(t) Q_{ ij}(t). \label{scheme} 
\end{align}
For grid points   $(i,j) \in \Nfluid $ in the  ``fluid'' domain $\Df$    Assumption \ref{assum2} implies that $v_{ij}^x=1$ while $v_{ij}^y=0$ and the above coefficients are given by
\begin{align}		
	\label{coeff_schem}
	\begin{array}{rl}
		\alpha_{ij}^{+}(t)&=\alpha^{F+}=\frac{\af }{h^2_x}, \quad \alpha_{ij}^{-}(t)=\alpha^{F-}(t)=\frac{\af }{h^2_x}+\frac{v_0(t)}{h_x}, \quad 
		\beta_{ij}^{\pm}(t)=\betaf=\frac{\af }{h^2_y},\\
		\gamma_{ij}(t)&=\gamaf (t)= -2\af \Big(\frac{1}{h^2_x}+\frac{1}{h^2_y}\Big)-\frac{v_0(t) }{h_x},\qquad \text{with }~\af =\frac{\kappaf}{\rhof \cpf}.			
	\end{array}			 
\end{align}
In the  ``medium'' domain $\Dm$ the convection terms disappear and the coefficients of the scheme (\ref{scheme}) become time-independent and are given for $(i,j) \in \Nmed ,$ by 
\begin{align}
	\label{coeff_schem_medium}
	\alpha^\pm_{ij}(t)=\alpham = \frac{\am }{h^2_x},  ~~ \beta^\pm_{ ij}(t)=\betam =\frac{\am }{h^2_y},~~\gamma_{ij}(t)=\gamam = -2\am \Big(\frac{1}{h^2_x}+\frac{1}{h^2_y}\Big),	
\end{align}
and $\am =\frac{\kappam}{\rhom \cpm}$.
Note that for the grid points in the neighborhood of the interfaces we have to slightly modify the above scheme (\ref{scheme}) due to the extra contribution from the interfaces, see equations (\ref{medium_I}) and (\ref{fluid_I}) below.

\subsection{Semi-Discretization of the Boundary Conditions}
In this paragraph we consider the discretization of boundary conditions. We start with the  homogeneous Neumann conditions \eqref{Neumann} and \eqref{output} for the top, left, right and the outlet boundary, where the normal vector  $\normalvec $ is equal to $(0,1)^\top, (-1,0)^\top$, $(1,0)^\top$ and $(1,0)^\top$, respectively. 
Using first-order differences for the normal derivative we obtain for all $t \in [0,T]$

\begin{align}
	\begin{cases}
		\begin{array}{rllcl}
			Q_{iN_y}(t)&=Q_{iN_y-1}(t) &\qquad \text{for } &(i,N_y) &\in \NtopB,	\\[0.5ex]
			Q_{0j}(t) &=Q_{1j}(t)   &\qquad  \text{for } &(0,j) &\in \NleftB,	\\[0.5ex]
			Q_{N_xj}(t)& =Q_{N_x-1j}(t) &\qquad  \text{for }& (N_x,j) &\in \NrightB  \cup \NoutB.	
		\end{array}
	\end{cases} \label{condition_1}
\end{align}
Next, we discretize the Robin condition \eqref{Robin} at the bottom boundary $\partial \Dbottom $. We have $\normalvec =(0,-1)^\top$ such that for all grid points  $ (i,0) \in \NbottomB $, we have for all $t \in [0,T]$
\begin{align}
	Q_{i0}(t)=\frac{\kappam}{\kappam+\heattransfer h_y} Q_{ i1}(t)+\frac{\heattransfer h_y}{\kappam+\heattransfer h_y}\Qg(t).  \label{condition_2}  
\end{align}
On the inlet boundary $\partial \Din $ we have  according to \eqref{input} a Dirichlet boundary condition during pumping and a Neumann condition if the pump is off. Then for all grid points $ (0,j) \in \NinletB$, we have  $\normalvec =(-1,0)^\top$ which implies for all $t \in [0,T]$ 
\begin{align}
	\begin{cases}
		Q_{0j}(t)	=\Qin(t), & \qquad \text{ if pump~on}, \\ 
		Q_{0j}(t)=Q_{1j}(t),	& \qquad \text{ if pump~~off}. 
	\end{cases} \label{condition_3}
\end{align}
The relations \eqref{condition_1} through \eqref{condition_3} represent linear algebraic equations which allow to express the grid values $Q_{ij}(t)$ in the boundary grid points  $(i,j) \in \mathcal{N}^{\mathcal{B}}$  in terms of the corresponding values in the neighboring points in the interior of the domain  and the input data to the boundary conditions. Thus, in the finite difference scheme  these values $Q_{ij}(t)$ can be removed from the set of unknowns.

\subsection{ Semi-Discretization of Interface Condition}
\label{interface_discrete}
Now we consider grid points on the interface $\DInterface$ between  fluid and medium which are by Assumption \ref{assum1} straight lines in $x$-direction. That interface can be decomposed as  $\DInterface=\DInterfaceL \cup \DInterfaceU$, with $\DInterfaceL$ and $\DInterfaceU$ representing the lower and upper interface, respectively, see Fig.~\ref{Interface1}. 
\begin{figure}[h!]
	\begin{center}
		\resizebox{0.6\textwidth}{0.18\textwidth}{%
			\begin{tikzpicture}[thick]
				\shade[top color=red!20, bottom color=red!20]
				(0,0) rectangle (8,1);
				
				\draw[thick] (0,0) -- (8.0,0) node[anchor= west] {~~$\DInterfaceL$~~ Lower interface};
				\draw[thick] (0,1) -- (8,1) node[anchor= west] {~~$\DInterfaceU$~~ Upper interface};
				\node at (1.8,0.5) {$(i,j+1)$};
				\node at (1,0.5) {$\bullet$};
				\node at (1.5,-0.25) {$(i,j)$};
				\node at (1,0) {$\bullet$};
				\node at (1.8,-0.7) {$(i,j-1)$};
				\node at (1,-0.5) {$\bullet$};
				
				\node at (6.8,0.4) {$(i,j-1)$};
				\node at (6,0.5) {$\bullet$};
				\node at (6.5,1.2) {$(i,j)$};
				\node at (6,1) {$\bullet$};
				\node at (6.8,1.7) {$(i,j+1)$};
				\node at (6,1.5) {$\bullet$};
				\node at (4,1.5) { Soil $~~~\kappam$};
				\node at (4,-.5) { Soil $~~~\kappam$};
				\node at (4,0.5) { Fluid $~~~\kappaf$};
			\end{tikzpicture}
		}	
	\end{center}
	\mycaption{\label{Interface1} Interface between the fluid and soil.}	
\end{figure}
We define the outer normal  by $\normalvec =(0,1)^\top$ on the  upper interface and by $\normalvec =(0,-1)^\top$ for lower interface. Note that we have $n_P$ \phxs and each \phx has two interfaces. Then, we have in total $2 n_P$ interface subdomains. 

For a grid point  $(x_i,y_j)$  on the interface $\DInterface$  the perfect contact condition \eqref{eq: 13f} implies that at a given time $t$ the temperature of the fluid $ \Qff (t,x_i,y_j)$   is equal to the temperature $\Qmm (t,x_i,y_j)$ of the medium at that point. { As usual, $Q_{ij}(t)$ denotes the semi-discrete approximation of that temperature. Then discretization of the interface condition (\ref{Interface}) leads to	
	\begin{align*}
		& \kappaf\frac{ Q_{ij}(t)-Q_{i,j+1}(t)}{h_y}=\kappam\frac{ Q_{i,j-1}(t)-Q_{ ij}(t)}{h_y} \quad~~\text{for lower interface} \nonumber\\ &\kappam\frac{Q_{i,j+1}(t)-Q_{ij}(t)}{h_y}=\kappaf\frac{Q_{ ij}(t)-Q_{i,j-1}(t)}{h_y} \quad~~\text{for upper interface}.
	\end{align*}
	We obtain the  following coupling between the grid values in an interface  grid point $(i,j) \in \Ninter $ and its  neighbors  in vertical direction a time  $ t \in [0,T]$,
	\begin{align}
		\nonumber
		Q_{ij}(t)&=\weightf  Q_{i,j+1}(t)+ \weightm  Q_{i,j-1}(t),~~\quad  (i,j) \in \NinterL,
		\\
		\label{interface_interpol} 			
		Q_{ij}(t)&=\weightf  Q_{i,j-1}(t)+ \weightm  Q_{i,j+1}(t), ~~\quad  (i,j) \in \NinterU,  \\
		\nonumber
		\text{where}\quad \weightf & =\frac{\kappaf}{\kappaf+\kappam}\quad  and  \quad \weightm =1-\weightf.
	\end{align}	
	The above relations show that  the grid values $Q_{ij}(t)$ in the interface grid points  $(i,j) \in \Ninter $ can be expressed as linear combinations  of the grid  values in the two vertical neighboring points in the fluid and medium. Thus, in the finite difference scheme  these values $Q_{ij}(t)$ can be removed from the set of unknowns.
	Now, let   $(i,j)\in \NinterL$ be an interface point on the lower interface. Then  substituting the above expressions for $Q_{ij}(t)$   into the finite differences scheme \eqref{scheme}  applied to the lower neighbor $(i,j-1)\in\Nmed $ in the medium leads to 
	
	\begin{align}
		\frac{d}{dt} Q_{i,j-1}(t) 
		& = \alpham  Q_{i+1,j-1}(t)+\alpham  Q_{i-1,j-1}(t)+\betam  Q_{i,j-2}(t) +\betam _I Q_{i,j+1}(t) + \gamam _I Q_{i,j-1}(t)	  
		\nonumber \\\
		\text{with}\quad \betam _I &  =\weightf \betam  \quad\text{and}\quad   \gamam _I= \gamma+\weightm \betam ,
		\label{medium_I}
	\end{align}
	whereas for the upper neighbor $(i,j+1)\in\Nfluid $ in the fluid it holds 
	\begin{align}
		\frac{d}{dt} Q_{i,j+1}(t)\!
		& = \alpha^{F+} \!Q_{i+1,j+1}(t)+\!\alpha^{F-} \!Q_{i-1,j+1}(t)\!+\betaf Q_{i,j+2}(t)\!+ \betaf_I Q_{i,j-1}(t)\! +\gamaf _I  Q_{i,j+1}(t) 
		\nonumber \\
		\text{with}\quad \betaf_I &  =\weightm \betaf \quad\text{and}\quad   \gamaf _I=\gamma+\weightf \betaf. \label{fluid_I}
	\end{align}
	Similar expressions can be derived for points  $(i,j)\in \NinterU$  on the upper interface.

	\subsection{Matrix Form of the Semi-Discrete Scheme}
	\label{sec: gs}
	
	We are now in a position to establish a semi-discretized version of the heat equation \eqref{heat_eq2} in terms of a system of ODEs by summarizing   relations (\ref{scheme}), (\ref{medium_I}) and (\ref{fluid_I}).
	To this end we recall that the temperature at the boundary grid  points can be obtained by the linear algebraic equations  \eqref{condition_1} through \eqref{condition_3} derived from the boundary conditions. Further,  the values at the interface points are obtained by the interpolation formulas in \eqref{interface_interpol} derived from the perfect contact condition. 
	Thus, we can exclude these grid points from the subsequent considerations where we collect the semi-discrete approximations of the temperature $Q(t,x_i,y_j)$ at the remaining points of the grid 
	in the vector function 
	$	Y(t)=(Y_1(t),Y_2(t),\ldots,Y_n(t))^T$.
	The enumeration of the entries of $Y$ is such that we start with the  first inner grid point $(1,1)$ next to the lower left corner of the domain. Then we number   grid points consecutively
	in vertical direction where we exclude the $2n_P$ points of the interfaces of the $n_P$ \phxs such that we have $\ncol=N_y-2n_P-1$ points in each ``column'' of the grid. Thus, 
	$Y_{(i-1)q+1}$ corresponds to grid point $(i,1)$ for $i=1,\ldots,N_x-1$, and the last entry $Y_n$ to the inner grid point $(N_x-1,N_y-1)$ next to the domain's upper right corner. The dimension of $Y$ is   $n=(N_x-1)\ncol=(N_x-1)(N_y-2n_P-1).$
	The enumeration described above  can be expressed formally by a mapping $\mathcal{K}:\Nfm \to \{1,\ldots,n \}$  with  $(i,j)\mapsto  l=\mathcal{K}(i,j)$ 
	which maps  pairs of indices $(i,j)$ of  grid point $(x_i,y_j) \in\Domainspace$ to the single index $l$ of the corresponding entry in the vector $Y$. 
	
	Using the above notations we can rewrite  relations (\ref{scheme}), (\ref{medium_I}) and (\ref{fluid_I}) as the  following system of ODEs for the vector function $Y$ representing the semi-discretized heat equation \eqref{heat_eq2} together with the given boundary and interface conditions.
	\begin{align}
		\frac{d Y(t)}{dt}= \mat{A}(t)Y(t)+\mat{B}(t)g(t), ~~t \in (0,T],
		\label{Matrix_form1}
	\end{align}
	with the initial condition $Y(0)=y_0$. Here, the vector $y_0\in \R^n$ contains the initial temperatures at the corresponding grid points with $y_{0l}=Q_0(x_i,y_j)$ where $l=\mathcal{K}(i,j), l=1\ldots,n$.   
	The system matrix  $\mat{A}$  results from the spatial discretization of the convection and diffusion term in the heat equation (\ref{heat_eq2}) together with the Robin and linear heat flux boundary conditions. It has tridiagonal structure  consisting of $(N_x-1)\times( N_x-1)$ block matrices of dimension $\ncol$ given by
	\begin{align}
		\label{matrix_A}
		\mat{A}=\begin{pmatrix}
			\mat{A}_{L} ~&~ \mat{D}^{+} ~& &&&\text{\LARGE0} \\
			\mat{D}^{-} ~&~ \mat{A}_{M} ~&~ \mat{D}^{+} \\
			& \mat{D}^- ~&~ \mat{A}_{M} ~&~ \mat{D}^{+} \\
			&& \ddots &\ddots & \ddots &\\
			& && \mat{D}^- ~&~ \mat{A}_{M}~ &~ \mat{D}^+ \\
			\text{\LARGE0}& & &&~ \mat{D}^- ~&~ \mat{A}_{R}
		\end{pmatrix}.
	\end{align}
	The inner block matrices $\mat{A}_{M}, i=2,\ldots N_x-2$ of dimension $\ncol$ have tridiagonal structure and are sketched for the case of one \phx in Table \ref{table:matrix}. 
	The matrix entries $\betaf,\gamaf $ are given in \eqref{coeff_schem}, $\betam ,\gamam $ in   \eqref{coeff_schem_medium},  $\betam _I,\gamam _I$ in \eqref{medium_I} and $\betaf_I,\gamaf _I$ in \eqref{fluid_I}. The first and last diagonal  entry   reads as
	$		
	\gamam _B= \gamam +\frac{\kappam }{\kappam+\heattransfer h_y}\betam  , \quad \gamam _T= \gamam +\betam,
	$		
	respectively. They  are obtained if the discretized top and bottom boundary conditions  \eqref{condition_1} and \eqref{condition_2} are substituted into \eqref{scheme}.	
	
	For  the matrices $\mat{A}_L$ and  $\mat{A}_R$ containing  entries resulting from the discretization of boundary conditions at the  left and right boundary  we refer to  Appendix \ref{append_block_matrices}.
	\begin{table}[h!]
		{\footnotesize
			\[
			~\mat{A}_{M}=\left(\begin{array}{ccccccccccccccc}
				\gamam _{B} & \betam  &&&&&&&&&&\multicolumn{3}{c}{\text{ \scriptsize Bottom Boundary }} & \\
				\betam  & \gamam & \betam &&&&&&&&&& \\
				&\ddots&\ddots&\ddots&&&&& &\text{\Large0}&&\multicolumn{3}{l}{\text{\scriptsize Medium}}\\
				&&\betam  & \gamam & \betam &&&&& \\
				&&&\betam  &\gamam _{I}	 &   \betam _I && &&&&\multicolumn{3}{l}{\text{\scriptsize Lower interface}} \\ \hline\\[-1.5ex]
				&&&& \betaf_I& \gamaf _{I}	 & \betaf & &&& \\
				&&&&&\betaf & \gamaf &~ \betaf&&\\
				&&&&&&\ddots&\ddots&\ddots& &&\multicolumn{3}{l}{\text{\scriptsize Fluid}}\\
				&&&&&&&\betaf & \gamaf & \betaf&& \\
				\multicolumn{3}{l}{\text{\scriptsize Upper interface}}&&&&&&\betaf & \gamaf _{I}	 &   \betaf_I&& \\ \hline \\[-1.5ex]
				&&&&&&&&& \betam_I&\!  \gamam _{I}	 & \betam &\\
				&&&&&&&&&&\betam  & \gamam & \betam & \\
				\multicolumn{3}{l}{\text{\scriptsize Medium}}&&&\text{\Large 0}&&&&&&\ddots&\ddots&\ddots&\\
				&&&&&&&&&&&& \betam  & \gamam & \betam  \\
				\multicolumn{3}{c}{\text{\scriptsize Top Boundary }}&&&&&& &&&& &\betam  &  \gamam _T
			\end{array} \right).\] 
		}		
		\mycaption{Sketch of inner block matrices $\mat{A}_{M}, i=2,\ldots,N_x-2$ for the case of one \phxk.}
		\label{table:matrix}
	\end{table}
	The lower and upper block matrices $\mat{D}^{\pm} \in \R^{\ncol \times \ncol}$, $i=1,\ldots, N_x-1$, are  diagonal matrices  of the form 
	\begin{align}
		\label{mat_D}
		\mat{D}^{\pm}=\mat{D}^{\pm}(t)=\diag{(\alpham ,\ldots,\alpham |\alpha^{F \pm}(t),\ldots,\alpha^{F \pm}(t)| \alpham ,\ldots,\alpham )},
	\end{align} 
	where $\alpham $ is given in \eqref{coeff_schem_medium} and $\alpha^{F \pm}$ in    \eqref{coeff_schem}.
	Here, we denote by $|$ the location of the interfaces where we only sketched  the case of one \phxk.   For the convenience of the reader we provide a comprehensive list of all  entries of matrix \mat{A} showing the dependence on model and discretization parameters in Appendix \ref{append_properties_A}.

	The   $n\times 2$ input matrix $\mat{B}$ is a result from the discretization of the  inlet and Robin boundary conditions, its entries $B_{lr}, ~l=1,\ldots,n,~~ r=1,2,$ are derived in Appendix \ref{append_block_matrices} and are given by 
	\begin{align}\label{eq:input_matrix}
		\begin{array}{rl@{\hspace*{2em}}l}
			B_{l1}&=B_{l1}(t)=\begin{cases}
				\frac{\af }{h^2_x}+\frac{\vconst}{h_x},  & \text{pump on,}\\
				0, & \text{pump off,}
			\end{cases}
			& l=\mathcal{K}(1,j), (0,j)\in\NinletB,\\[3ex]	
			B_{l2}& = \frac{\heattransfer h_y}{\kappam+\heattransfer h_y}\betam ,
			&  l=\mathcal{K}(i,1), (i,0)\in\NbottomB.
		\end{array}		
	\end{align}   
	The entries for other $l$ are zero.
	The input function  $g:~[0,T] \to \R^2$  is defined by 
	\begin{align}
		g(t)=\begin{cases}
			(\Qin(t),~\Qg(t))^{\top}, & \quad \text{pump on},\\
			~~~~~~~(0,~\Qg(t))^{\top}, & \quad \text{pump off}.
		\end{cases}
		\label{eq:input}
	\end{align}
	Recall that  $\Qin$ is the inlet temperature  of the \phx during pumping  and $\Qg$ is the underground temperature.		
	\subsection{Stability of Matrix $\mat{A}$}
	\label{stability_A}	
	The finite difference  semi-discretization of the heat equation \eqref{heat_eq2} given by the  system of ODEs \eqref{Matrix_form1} is expected to preserve the dissipativity of the PDE. This property is related to the stability of the system matrix $\mat{A}=\mat{A}(t)$  in the sense that all eigenvalues  of $\mat{A}$ lie in the left open complex half plane. That property will play a crucial role   for model reduction techniques for \eqref{Matrix_form1} based on balanced truncationin which we study in  \cite{Takam2020Reduction}.	The next theorem confirms the expectations on the stability of $\mat{A}$.	
	\begin{theorem}[Stability of Matrix $\mat{A}$]	\label{stable_m} \ \\
		Under Assumption \ref{assum1} on the model and Assumption \ref{assum2} on the discretization,  
		the matrix $\mat{A}=\mat{A}(t)$ given in \eqref{matrix_A} is stable for all $t\in[0,T]$, i.e.,  all eigenvalues $\lambda (\mat{A})$ of $\mat{A}$ lie in   left open complex half plane. 
	\end{theorem}
	
	\begin{proof}
		Lemma \ref{lem_eigvalueA} in Appendix \ref{append_properties_A}  shows  by using  Gershgorin's circle theorem, that the eigenvalues are either located in   left open complex half plane or zero. Further, Lemma \ref{lem_nonsingular} (also in Appendix  \ref{append_properties_A}) shows that $\mat{A}(t)$ is non-singular for all $t\in[0,T]$ and thus excludes the case $\lambda(\mat{A})=0$.	
		Thus, for all eigenvalues it holds  $\lambda (\mat{A})\in \mathbb{C}_-$ and $\mat{A}$ is stable.  
		\qed
	\end{proof}

	\section{Full Discretization}
	\label{sec:full_discrtization}
	After discretizing the heat equation \eqref{heat_eq2} w.r.t.~spatial variables  we will now also discretize the temporal derivative and derive  an explicit and a family of implicit finite difference schemes for which we perform a stability analysis. 
	\subsection{ $\theta$-Implicit Finite Difference Scheme}\label{implicit:scheme}

	We introduce the notation  $N_{\tau}$ for  the number of grid points in $t$-direction,   $\tau={T}/{N_{\tau}}$ the time step and $t_k=k \tau, ~k \in \mathcal{N}_{\tau}=\{0,...,N_{\tau}\}$. Let $\mat{A}^k,\mat{B}^k, g^k, v_0^k$ be the values of $\mat{A}(t),\mat{B}(t)$, $g(t), v_0(t)$ at time  $t=t_k$, respectively. Further, we denote by 
	$ Y^k=(Y^k_1,\ldots,Y^k_n)^\top $ the discrete-time approximation of the vector function $Y(t)$ at time  $t=t_k$. Recall that $Y$ contains the temperatures $Q=Q(t,x,y)$ at the points of the grid excluding points on the boundary and interface.
	Discretizing the temporal derivative in \eqref{Matrix_form1} with the forward difference gives 
	\begin{align}
		\frac{d Y(t_k)}{d t} = \frac{Y^{k+1}-Y^k}{\tau}+ \mathcal{O}(\tau).
		\label{time_d}
	\end{align} 
	Substituting \eqref{time_d} into \eqref{Matrix_form1} and replacing the r.h.s.~of \eqref{Matrix_form1} by a convex combination of the values at time $t_k$ and $t_{k+1}$ with the weight  $\theta \in [0,1]$ gives the following general $\theta$-implicit finite difference scheme 
	\begin{align}
		\frac{Y^{k+1}-Y^k}{\tau} = \theta[\mat{A}(t_{k+1})Y^{k+1} +\mat{B}(t_{k+1}) g^{k+1}] +(1-\theta)[\mat{A}(t_{k})Y^{k} +\mat{B}(t_{k}) g^{k}]
		\nonumber
	\end{align}
	from which we derive for $k=0,\ldots,N_\tau-1$ the recursion
	\begin{align}
		\label{Matrix_form2} 
		\mat{G}^{k+1}Y^{k+1}& =\mat{H}^{k} Y^{k}+\tau {F}^{k} \\
		\nonumber
		\text{where }~\mat{G}^{k}&=\mathds{I}_{n}-\tau \theta \mat{A}^{k}, ~~\mat{H}^{k}=\mathds{I}_{n}+\tau(1-\theta)\mat{A}^{k}, ~\text{and} ~{F}^{k}=\theta \mat{B}^{k+1}g^{k+1}+(1-\theta)\mat{B}^{k}g^{k},
	\end{align}
	with the initial value $Y^0=Y(0)$ and the notation $\mathds{I}_{n}$ for the  $n \times n$ identity matrix.
	
	The above general  $\theta$-implicit scheme leads  for $\theta= 1/2$ and  $1$  to special cases which are known in the literature as Crank-Nicolson scheme   and backward Euler or  fully implicit scheme, respectively. The limiting case $\theta=0$ is not an implicit but a  fully explicit scheme  also known as forward Euler  scheme.
	
	\subsection{Stability of the Finite Difference Scheme}
	\label{stability_num_scheme}
	
	In this subsection we investigate the stability of the finite difference scheme \eqref{Matrix_form2} in the maximum norm and present in Theorem \ref{stability_scheme} below a stability condition to the time discretization. The use of such stability results is twofold.  First, it ensures ``robustness'' w.r.t. round-off errors of the problems's input data, which are the  initial condition and the inlet and underground temperature, in the sense that  we  can run the recursion for an arbitrarily long time without a total loss of accuracy. 
	Second, stability of the scheme is a key ingredient in any analysis of convergence of the exact solution of the finite difference scheme  to the exact solution of the given initial boundary value problem for the PDE for an infinite refinement of space and time discretization.
	
	Note that a complete convergence analysis is beyond the scope of this paper. In particular, we do not investigate consistency  issues. Consistency roughly says that the finite differences scheme approximates correctly the PDE.	The proof of consistency is straightforward and based on Taylor series expansions. 
	We refer to the Lax-Richtmyer Equivalence Theorem, see Sanz-Serna and Palencia \cite{sanz1985general}, 
	Thomas \cite[Theorem 2.5.3]{thomas2013numerical}, stating that a consistent finite difference scheme for a well-posed linear initial boundary value problem,  is convergent if and only if it is stable. Hence, for a consistent scheme, convergence is synonymous with stability. 
	
	Our stability result is given in terms of  maximum norms which are defined for  a vector $X \in \R^n$  by 
	$\|X\|_{\infty}=\displaystyle\max_{1\leq i\leq n}|X_i|$  and for a square matrix $\matgen \in \C^{n \times n}$ by  $\|\matgen\|_{\infty}=\displaystyle\max_{1\leq i\leq n} \sum_{j=1}^{n}|\matgenel_{ij}|$.
	
	\begin{definition}[Stability of difference scheme in the maximum norm]
		\label{def_stability}\\
		The finite difference scheme (\ref{Matrix_form2}) is stable in the  maximum norm if there exist constants $C_0, ~C_g>0$ such that 
		\begin{align}
			\label{def:stability}
			\|Y^{k}\|_{\infty} \leq C_0\|Y^0\|_{\infty}+ C_g \max_{0\leq j\leq k}\|g^j\|_{\infty} ~~\text{for}~~k=1,2,\ldots, \mathcal{N}_{\tau}.
		\end{align}
	\end{definition}
	
	\begin{theorem}[Stability of $\theta$-implicit scheme]
		\label{stability_scheme}\\
		Under Assumption \ref{assum1} on the model and Assumption \ref{assum2} on the discretization it holds	
		\begin{enumerate}
			\item For $\theta \in [0, 1)$, the semi-implicit finite difference scheme (\ref{Matrix_form2}) is stable if the time step $\tau$ satisfies the condition 
			\begin{align}
				\tau \leq \frac{1}{(1-\theta) \eta},  \label{stability}
				\quad \text{where}\quad 
				\eta			=2 \max\{\af ,\am \}\Big(\frac{1}{h_x^2}+\frac{1}{h_y^2}\Big)+\frac{\vconst }{h_x}.
			\end{align}
			
			\item For $\theta =1$, the fully implicit finite difference scheme (\ref{Matrix_form2}) is unconditionally stable, i.e., stable for any $\tau>0.$			
		\end{enumerate}
		The constants $C_0,C_g$ in \eqref{def:stability} can be chosen as 
		\begin{align}
			\label{bound_CB}
			C_0=1 ~\text{ and }~ C_g= C_B T \quad \text{where}\quad C_B=\max \big\{\big\|\mat{B}^{P}\big\|_{\infty},~\big\|\mat{B}^{N}\big\|_{\infty}\big\}.
		\end{align}
	\end{theorem}
	The proof is based on the following lemma which is proven in Appendix \ref{append_stability_scheme}.
	\begin{lemma} \label{lem:bounds_GHF}	Under Assumption \ref{assum1} on the model and Assumption \ref{assum2} on the discretization it holds for all $k=0,\ldots,N_\tau-1$ and $\theta\in[0,1]$ that
		\begin{enumerate}
			\item  the matrices $\mat{G}^{k+1}$ given in \eqref{Matrix_form2} are invertible and $	\|(\mat{G}^{k+1})^{-1}\|_\infty\le 1$ with equality for $\theta=0$; 
			\item
			the matrices $\mat{H}^k$ given in \eqref{Matrix_form2} satisfy $\|\mat{H}^k\|_\infty\le 1$ ~for all $\tau>0$ if $\theta=1$;\\ and for 
			$  \tau \leq \frac{1}{(1-\theta) \eta}$ if $\theta\in[0,1)$, where $\eta$ is given in \eqref{stability};
			\item the vectors $F^k$  given in \eqref{Matrix_form2} satisfy 	$\|{F}^k\|_\infty\le C_B\displaystyle\max_{0\leq j\leq k+1}\big\|g^{j}\big\|_{\infty}$			
			where $C_B$ given in \eqref{bound_CB}. 
		\end{enumerate}	
	\end{lemma}	
	\begin{proof} of Theorem \ref{stability_scheme}.
		From  the invertibility of $G^k$ (see Lemma \ref{lem:bounds_GHF},1.) and the iteration of the recursion \eqref{Matrix_form2} we obtain   for $~k=1,\ldots, \mathcal{N}_{\tau}$ the explicit representation
		\begin{align*}
			Y^{k}&=(\mat{G}^{k})^{-1}\mat{H}^{k-1}Y^{k-1}+\tau(\mat{G}^{k})^{-1}{F}^{k-1}\\
			&=(\mat{G}^{k})^{-1}\mat{H}^{k-1}(\mat{G}^{k-1})^{-1}\mat{H}^{k-2}Y^{k-2}+\tau(\mat{G}^{k})^{-1}\mat{H}^{k-1}(\mat{G}^{k-1})^{-1}{F}^{k-2}+ \tau(\mat{G}^{k})^{-1}{F}^{k-1} \nonumber\\
			&=\ldots=\Big(\prod_{j=1}^{k}(\mat{G}^{k-j+1})^{-1}\mat{H}^{k-j} \Big)Y^0+\tau\sum_{j=0}^{k-1} \Big(\prod_{i=1}^{j}(\mat{G}^{k-i+1})^{-1}\mat{H}^{k-i} \Big) (\mat{G}^{k-j})^{-1}{F}^{k-j-1},
		\end{align*}
		where we define $\displaystyle\prod_{j=1}^{0}(\cdot)=\mathds{I}_n$. 
		Taking the maximum norm on both sides 
		and applying the triangular and Cauchy-Schwarz inequality  gives
		\begin{align*}
			\big\|Y^{k}\big\|_{\infty} &\leq \Big(\prod_{j=1}^{k}\big\|(\mat{G}^{k-j+1})^{-1}\big\|_{\infty}\big\|\mat{H}^{k-j}\big\|_{\infty} \Big)\big\|Y^0\big\|_{\infty} \\& 
			~~~~+ \tau\sum_{j=0}^{k-1} \Big(\prod_{i=1}^{j}\big \|(\mat{G}^{k-i+1})^{-1}\big\|_{\infty}\big\|\mat{H}^{k-i} \big\|_{\infty} \Big) \big\|(\mat{G}^{k-j})^{-1}\big\|_{\infty}\big\|{F}^{k-j-1}\big\|_{\infty}.
		\end{align*}
		Substituting the estimates for $\|(\mat{G}^{k})^{-1}\|_{\infty}, \|\mat{H}^{k}\|_{\infty}$ and  $\|{F}^{k}\|_{\infty}$ given in Lemma \ref{lem:bounds_GHF} into the above inequality  yields
		\begin{align*}
			\big\|Y^{k}\big\|_{\infty} 	&\leq \big\|Y^0\big\|_{\infty} +\tau k\, C_B\displaystyle\max_{0\leq j\leq k}\big\|g^{j}\big\|_{\infty}
			\le \big\|Y^0\big\|_{\infty} +C_B T\max_{0\leq j\leq k}\big\|g^{j}\big\|_{\infty},
		\end{align*}
		where we used   $\tau k\le \tau N_\tau=T.$ According to the second assertion of Lemma \ref{lem:bounds_GHF}  the above estimate holds for all $\tau>0$ if  $\theta=1$ and   for  $\tau \leq \frac{1}{(1-\theta) \eta}$ if $\theta\in[0,1)$. 
		\qed
	\end{proof}
	
	\section{Conclusion}
	
	We have investigated numerical methods for the  simulation of the short-term behavior of the spatial temperature distribution in a geothermal energy storage. The underlying initial boundary value problem for the heat equation  with a convection term  has been discretized using finite difference schemes. 
	In a first step we studied the semi-discretization with respect to spatial variables. 
	For the resulting system of linear ODEs we proved that the system matrix is stable. In a second step the full space-time discretization has been considered.  Here we derived explicit and implicit finite-difference schemes and investigated associated stability problems.		
	
	Based on the findings of this paper we present in  \cite{Takam2021NumResults}  results of a large number of numerical experiments where we have shown how these simulations can support the design and operation of a geothermal storage. Examples are the dependence of the charging and discharging efficiency on the topology and arrangement  of \phxs and on the length of charging and discharging periods.   
	
	In \cite{Takam2020Reduction} we study model reduction techniques  to derive low-di\-men\-sional approximations of aggregated characteristics of the temperature distribution  describing the input-output behavior of the storage. The latter is crucial if the geothermal storage is embedded into a  residential heating system  and the cost-optimal management of such systems is studied mathematically in terms of optimal control problems.

	\newpage \clearpage
	\begin{appendix}	
		\section{List of Notations}
		\label{append_a}
		\begin{longtable}{p{0.3\textwidth}p{0.68\textwidth}l}
			$Q=Q(t,x,y)$ & temperature in the geothermal storage &\\
			$T$ & finite time horizon&\\
			$l_x$,~$l_y$,~$l_z$ &width, height and depth of the storage &\\
			$\Domainspace =(0, l_x) \times (0,l_y)$ &domain of the geothermal storage &\\
			$\Dm , ~\Df $ &  domain of  medium (soil) and  \phx fluid &\\
			$\DInterface=\DInterfaceL \cup \DInterfaceU$ & interface between the \phxs and the medium &\\
			$\partial \Domainspace$ &boundary of the domain&\\
			$\partial \Din $,~$\partial \Dout$ & inlet and outlet boundaries of the \phx&\\
			$\partial \Dleft , \partial \Dright , \partial \Dtop $,~$\partial \Dbottom $  & left, right, top and bottom boundaries of the domain&\\		
			$\Nmed,~\Nfluid $&  subset of index pairs of points  in the medium and \phx fluid&\\
			$\Ninter, \NinterL,\NinterU, \mathcal{N}^{\mathcal{B}}_*$  &  subsets of index pairs for points on  interface and  boundary &\\
			$\mathcal{K}$  &  mapping $(i,j)\mapsto l=\mathcal{K}(i,j)$ of index pairs to single indices&\\
			$v=v_0(t)(v^x, v^y)^{\top}$ & time-dependent velocity vector,&\\
			$\vconst$ &  constant velocity during pumping&\\
			$\cpf$,~$\cpm$ & specific heat capacity of the fluid  and medium &\\
			$\rhof$,~$\rhom$ & mass density of the fluid and medium&\\
			$\kappaf$,~$\kappam$ &thermal conductivity of the fluid and medium&\\
			$\af $,~$\am $& thermal diffusivity of the fluid and medium&\\
			$\heattransfer$ & heat transfer coefficient between storage and  underground &\\
			$Q_0$ &initial temperature distribution of the geothermal storage &\\
			$\Qg(t)$ & underground temperature &\\
			$\Qin(t), \QinC(t), \QinD(t)$ & inlet temperature of the \phxk, during charging and discharging, &\\
			$N_x,~N_y$,~$N_{\tau}$ &  number of grid points in $x,y$ and $\tau$-direction&\\
			$h_x, h_y$,~$\tau$ & step size in $x$ and $y$-direction and the time step&\\
			$\normalvec $& outward normal to the boundary $\partial\Domainspace$&\\
			$n$ &   dimension of vector $Y$ &\\	
			$n_P$ & number of \phxs&\\	
			$\mathds{I}_n$& $n \times n$ identity matrix&\\
			$\mat{A}$ & $n \times n$ dimensional  system matrix &\\
			$\mat{B}$ & $n \times m$ dimensional input matrix    &\\
			$\mat{D}^{\pm}, ~\mat{A}_{L}, ~\mat{A}_{M},~\mat{A}_{R}$ & block matrices of matrix $\mat{A}$&\\
			$q$ &   dimension of block matrices  &\\
			$~\alpha^*_*,~\beta^*_*,~\gamma^*_*$&coefficients of the matrix $\mat{A}$ &\\
			$\weightf ,\weightm $ & weighting factors for discretitzation of interface condition& \\	
			$Y$&vector of temperatures at grid points&\\
			$g$& input variable of the system&\\
			$\nabla$, ~~$\Delta=\nabla \cdot \nabla$ & 	 gradient, Laplace operator&\\
			\phx &  pipe heat exchanger
		\end{longtable}

		\section{Block  Matrices $\mat{A}_L$ and $\mat{A}_R$}
		\label{append_block_matrices}       
		
		\begin{table}[h!]
			
			\[
			~\mat{A}_{L/R}=\left(\begin{array}{ccccccccccccccc}
				\gamam _{BB} & \betam  &&&&&&&&&&\multicolumn{3}{c}{\text{ \scriptsize Bottom Boundary }} & \\
				\betam  & \gamam _B& \betam &&&&&&&&&& \\
				&\ddots&\ddots&\ddots&&&&& &\text{\Large0}&&\multicolumn{3}{l}{\text{\scriptsize Medium}}\\
				&&\betam  & \gamam _B& \betam &&&&& \\
				&&&\betam  &\gamam _{IB}	 &   \betam _{I} && &&&&\multicolumn{3}{l}{\text{\scriptsize Lower interface}} \\ \hline\\[-1.5ex]
				&&&& \betaf_I& \gamaf _{IL/R}	 & \betaf & &&& \\
				&&&&&\betaf & \gamaf _{L/R}&~ \betaf&&\\
				&&&&&&\ddots&\ddots&\ddots& &&\multicolumn{3}{l}{\text{\scriptsize Fluid}}\\
				&&&&&&&\betaf & \gamaf _{L/R}& \betaf&& \\
				\multicolumn{3}{l}{\text{\scriptsize Upper interface}}&&&&&&\betaf & \gamaf _{IL/R}	 &   \betaf_I&& \\ \hline \\[-1.5ex]
				&&&&&&&&& \betam_I&\!  \gamam _{IB}	 & \betam &\\
				&&&&&&&&&&\betam  & \gamam _B& \betam & \\
				\multicolumn{3}{l}{\text{\scriptsize Medium}}&&&\text{\Large 0}&&&&&&\ddots&\ddots&\ddots&\\
				&&&&&&&&&&&& \betam  & \gamam _B& \betam  \\
				\multicolumn{3}{c}{\text{\scriptsize Top Boundary }}&&&&&& &&&& &\betam  &  \gamam _{TB}
			\end{array} \right).\] 
			
			\mycaption{Sketch of the matrices $\mat{A}_{L}$ and  $\mat{A}_{R}$ for the case of one \phxk.}
			\label{table:matrix_last}
		\end{table}
		
		This appendix gives  the first and the last diagonal block matrices $\mat{A}_L$ and  $\mat{A}_R\in \R^{\ncol \times \ncol}$ of the matrix $\mat{A}$ given in \eqref{matrix_A}. Its entries result from the discretization of boundary conditions at the  left and right boundary.
		Both block matrices are tridiagonal and sketched for the case of only one \phx in Table  \ref{table:matrix_last}.
		The entries in the  first and last row  are related to the inner grid points next to the four corners of the domain and  obtained by substituting homogeneous Neumann condition \eqref{condition_1}  and  Robin condition   \eqref{condition_2} into \eqref{scheme}.
		For the grid points next to the lower left  we obtain 
		\begin{align*}
			\frac{d}{dt}Q_{11}(t) &= \alpham  Q_{21}(t)+\alpham  Q_{01}(t)+\betam  Q_{12}(t) +\betam  Q_{10}(t)+\gamam  Q_{11}(t) \\
			&=  \alpham  Q_{21}(t)+\betam  Q_{12}(t) + \bigg(\alpham +\frac{\kappam}{\kappam+\heattransfer h_y}\betam +\gamam \bigg) Q_{11}(t)+ \frac{\heattransfer h_y}{\kappam+\heattransfer h_y}\betam \Qg(t)\\
			&=\alpham  Q_{21}(t)+\betam  Q_{12}(t) + \gamam_{BB} Q_{11}(t)+ \frac{ \heattransfer h_y}{\kappam+\heattransfer h_y}\betam \Qg(t),
		\end{align*}
		where $\gamam_{BB}=\alpham +\frac{\kappam }{\kappam+\heattransfer h_y}\betam +\gamam $. Recall, that $\alpham , \betam ,\gamam $ are given by \eqref{coeff_schem_medium}. 
		
		Analogously, we derive for the lower right corner
		\begin{align*}
			\frac{d}{dt}Q_{N_x-1,1}(t) 
			&=  \alpham  Q_{N_x-2,1}(t)+\betam  Q_{N_x,2}(t) + \gamam_{BB} Q_{N_x,1}(t)+ \frac{\heattransfer h_y}{\kappam+\heattransfer h_y}\betam \Qg(t).
		\end{align*}
		Note that the last terms on the r.h.s.~of the above equations are contributions to the input term $\mat{B}(t) g(t)$  given in \eqref{eq:input_matrix} and \eqref{eq:input}. For the grid points next to the upper left and right corner we have to apply the homogeneous Neumann conditions \eqref{condition_1} and obtain from \eqref{scheme}
		\begin{align*}	
			\frac{d}{dt}Q_{1,N_y-1}(t) &= \alpham  Q_{2,N_y-1}(t)+\alpham  Q_{0,N_y-1}(t)+\betam  Q_{1,N_y}(t) + \betam  Q_{1,N_y-2}(t)+ \gamam  Q_{1,N_y-1}(t)\\
			&= \alpham  Q_{2,N_y-1}(t)+\betam  Q_{1,N_y-2}(t) +(\alpham +\betam +\gamam ) Q_{1,N_y-1}(t)\\
			&= \alpham  Q_{2,N_y-1}(t)+\betam  Q_{1,N_y-2}(t) + \gamam_{TB} Q_{1,N_y-1}(t),
		\end{align*}	
		where	$\gamam_{TB}=\alpham +\betam +\gamam $. Analogously, we derive for the upper right corner
		\begin{align*}	
			\frac{d}{dt}Q_{N_x-1,N_y-1}(t) 
			&= \alpham  Q_{N_x-2,N_y-1}(t)+\betam  Q_{N_x-1,N_y-2}(t) + \gamam_{TB} Q_{N_x-1,N_y-1}(t).
		\end{align*}	
		For ``inner'' grid points located next to insulated left and right boundary but not next to the upper and lower boundary or the interface we have to combine \eqref{scheme} with  the homogeneous Neumann condition \eqref{condition_1}. This leads to the  coefficient $\gamam_B =\gamam +\alpham $ on the main diagonal.
		
		For the grid points next to the inlet boundary  we apply  Dirichlet condition during pumping and  homogeneous Neumann condition if the pump is off, see  \eqref{condition_3}.    For $j$ with $(0,j)\in\NinletB$ it holds  
		\begin{align*}	
			\frac{d}{dt}Q_{1j}(t) &= \alpha^{F+} Q_{2j}(t)+\alpha^{F-} Q_{ 0j}(t)+\betaf Q_{1,j+1}(t) +\betaf Q_{1,j-1}(t)+\gamaf Q_{1j}(t) \\
			&=\alpha^{F+} Q_{2j}(t)+\betaf Q_{1j+1}(t) +\betaf Q_{1,j-1}(t) +\begin{cases}
				\gamaf Q_{1j}(t)+\alpha^{F-} \Qin(t) &\text{pump on}\\[0.3ex]
				(\gamaf +\alpha^{F-})Q_{1j}(t)  & \text{pump off}
			\end{cases}\\[0.5ex]
			&=\alpha^{F+} Q_{2j}(t)+\betaf Q_{1,j+1}(t) +\betaf Q_{1,j-1}(t)+\gamaf _LQ_{1j}(t)+b_{k1} \Qin(t),
		\end{align*}		
		where~~$\gamaf_{L}=\gamaf_{L}(t)=
		\begin{cases}
			\gamaf  & \text{pump on},\\
			\gamaf + \frac{\af }{h^2_x}& \text{pump off},
		\end{cases} $ ~~and~~ 
		$ B_{l1}=B_{l1}(t)=\begin{cases}\begin{array}{cl}
				\frac{\af }{h^2_x}+\frac{\vconst}{h_x}  & \text{pump on,}\\
				0 & \text{pump off,}
		\end{array}\end{cases}$ \\
		with $l=\mathcal{K}(1,j)$.	We note that 
		$\alpha^{F \pm}, \betaf, \gamaf $ are given in   \eqref{coeff_schem} and 
		point out that the term $b_{k1} \Qin(t)$ contributes to the input term  $\mat{B}(t) g(t)$.

		At the outlet boundary we have homogeneous Neumann condition and for the grid points next to the outlet  we obtain from the discretized boundary condition
		\eqref{condition_1}    for $j$ with $(N_x,j)\in\NoutB$ it holds
		\begin{align*}	
			\frac{d}{dt}Q_{N_x-1,j}(t) &= \alpha^{F+} Q_{N_x,j}(t)+\alpha^{F-} Q_{N_x-2,j}(t)+\betaf Q_{N_x-1,j+1}(t) +\betaf Q_{N_x-1,j-1}(t)+\gamaf Q_{N_x-1,j}(t) \\
			&=\alpha^{F-} Q_{N_x-2,j}(t)+\betaf Q_{N_x-1,j+1}(t) +\betaf Q_{N_x-1,j-1}(t)+\gamaf _RQ_{N_x-1,j}(t),
		\end{align*}
		where $	\gamma_{R}^f=\gamaf +\alpha^{F+}$ and  $\alpha^{F \pm}, \betaf, \gamaf $ are given in   \eqref{coeff_schem}.
		
		Finally, for the grid points next to the interface we obtain by an analogous procedure as described in Subsec.~\ref{interface_discrete} the coefficients
		\begin{equation*}
			\gamam _{IB}=\gamam _B+\weightm  \betam , \quad 
			\gamaf _{IL}=\gamaf _{IL}(t)=\gamaf _{L}(t)+\weightf  \betaf,\quad 
			\gamaf _{IR}=\gamaf _{R}+\weightf  \betaf,  
		\end{equation*}	
		where $\weightm $ and $\weightf $ are given  \eqref{interface_interpol}. Recall that the off-diagonal coefficients $\betam _I,\betaf_I$ are given in \eqref{medium_I},  \eqref{fluid_I}.

		\section{Auxiliary Results From Matrix Analysis}
		\label{append_matrix_analysis}
		In this appendix we collect some results from matrix analysis taken from the literature. They will be used in the proofs of Theorem \ref{stable_m} and Lemma \ref{lem:bounds_GHF}. Let  $\matgen \in \C^{n \times n}$ be some generic matrix.  For $i=1,\ldots,n$ we introduce the notations
		\begin{align}
			\label{mat:row:char}
			R_i(\matgen) &=\sum_{j \neq i}|\matgenel_{ij}|, \quad J_i(\matgen) =|\matgenel_{ii}| - R_i(\matgen), \quad 
			S_i(\matgen) =|\matgenel_{ii}| + R_i(\matgen)=\sum_{j }|\matgenel_{ij}|. 
		\end{align}
		Note that the maximum norm of $\matgen$ is given by  $\|\matgen\|_\infty=\max_i S_i(\matgen)$. The quantities $R_i(\matgen)$ appear as radii of Gershgorin's circles of $\matgen$ and the $J_i(\matgen)$ are used to describe diagonal dominance of  $\matgen$.
		\begin{lemma}[Gershgorin's Circle Theorem, Varga \cite{varga2004}] \label{Gershgorin_C}\\	
			Let $\matgen \in \mathbb{C}^{n \times n}$ and  for $ i=1,\ldots,n$ let $D_i=\{z \in \mathbb{C}: |z-\matgenel_{ii}| \leq R_i\}$ be the closed discs in the complex plane centred at $\matgenel_{ii}$ with radius $R_i=R_i(\matgen)$ given in  \eqref{mat:row:char}.
			Then all the eigenvalues of $\matgen$ lie in the union of  the discs $D_1,\ldots, D_n$. 
		\end{lemma}
		\begin{definition}[Diagonal Dominance]\label{def:diag:dom}\\
			Row $i\in\{1,\ldots,n\}$ of a matrix $\matgen \in \C^{n \times n}$ is called 
			\begin{tabular}[t]{cll}
				strictly &diagonal dominant  & if $ J_i(\matgen)>0$,\\
				weakly  &diagonal dominant  & if $J_i(\matgen)\geq 0$,
			\end{tabular}	\\
			The matrix $\mat{M}$ is called strictly  (weakly) diagonal dominant if all of its rows are strictly  (weakly) diagonal dominant.
		\end{definition}
		The following result says that strictly diagonal dominant matrices are invertible and provides a upper bound for the maximum norm of the inverse.
		\begin{lemma}[Varah \cite{varah1975lower}, Theorem 1]\\
			\label{Inverse}
			Let $\matgen \in \mathbb{C}^{n\times n}$  strictly diagonal dominant matrix. Then $\matgen$ is invertible and 
			\begin{align*}
				\big\|\matgen^{-1}\big\|_{\infty} \leq \frac{1}{J(\matgen)},\quad \text{where} \quad J(\matgen)=\displaystyle\min_{1\leq i \leq n}J_i(\matgen).
			\end{align*}
		\end{lemma}
		Matrices which are weakly but not strictly  diagonal dominant can be singular. A criterion for 
		non-singularity is based  on the following property  of a matrix and the subsequent lemma. 
		That property was introduced  in Horn and  Johnson \cite[Definition 6.2.7]{johnson1985matrix} and termed ``property SC''. In the literature it is also known as ``strongly connected''.
		\begin{definition}[Strongly Connected Matrix]\label{def:SC}\\
			A matrix $\matgen\in \mathbb{C}^{n \times n}$ is called  strongly connected (or of property SC) if for each pair
			of distinct integers $p, q \in \{1, \ldots,n\}$ there is a sequence of distinct integers $k_1 = p,k_2, . . ., k_m = q$ such that each entry $\matgenel_{k_1k_2}, \matgenel_{k_2k_3},$ \ldots,$ \matgenel_{k_{m-1}k_m}$  is non-zero.
		\end{definition}
		For strongly connected  matrices  Horn and Johnson \cite[Corollary 6.2.9]{johnson1985matrix} give the following criterion for non-singularity.
		\begin{lemma}[Better's Corollary]	\label{better Corr}\\
			Suppose that the matrix $\matgen\in \mathbb{C}^{n \times n}$ is strongly connected, weakly diagonally dominant and there exists one strictly diagonal dominant row. Then $\matgen$ is nonsingular.
		\end{lemma}	
		\section{Properties of Matrix $\mat{A}$}
		\label{append_properties_A}					
		We recall that the time-dependence of $\mat{A}(t)$ is  a result
		of the discretization of convection terms in the  heat equation \eqref{heat_eq2}. The latter depend  on the  time-dependent velocity $v_0(t)$ and with some abuse of notation we can write  $\mat{A}=\mat{A}(t)= \mat{A}(v_0(t))$ and  $\mat{B}=\mat{B}(t) =\mat{B}(v_0(t))$.
		Recall that we assume in Ass.~\ref{assum1} that  $v_0(t)$ is piecewise constant with $v_0(t)=\vconst$ during charging and discharging when the pump is on whereas $v_0(t)=0$ if the pump is off. Therefore,  the matrices  $\mat{A,B}$ share this property. They take only the two values $\mat{A}^P=\mat{A}(\vconst),\, \mat{B}^P=\mat{B}(\vconst)$ during pumping and $\mat{A}^N=\mat{A}(0),\, \mat{B}^N=\mat{B}(0)$ if the pump is off.
		Thus, for studying properties of $\mat{A}(t)$ on $[0,T]$ or of $\mat{A}^k=\mat{A}(k\tau)$ for $k=0,\ldots,N_\tau$ it is sufficient to look at the properties of $\mat{A}^P$ and $\mat{A}^N$.
		
		We want to have a closer look to the entries of the block matrices  $\mat{A}_M, \mat{A}_L, \mat{A}_R$ given in  Tables \ref{table:matrix}, \ref{table:matrix_last}   and of $\mat{D}^\pm$ given in  \eqref{mat_D}, forming the system matrix $\mat{A}$. It turns out that for the diagonal entries and the row characteristics $R_i, J_i, S_i$ given in \eqref{mat:row:char} one has to distinguish 14 different cases.  Instead of $n$ rows it is sufficient to consider only 14 representative rows whose indices we denote by $i_l, l=1,\ldots,14$. Table \ref{table_Rowsum} provides a list of diagonal entries $A_{i_li_l}$ and the row characteristics $R_{i_l}(\mat{A}), J_{i_l}(\mat{A}), S_{i_l}(\mat{A})$  in terms of the model and discretization parameters. For the convenience of the reader we give below that information also for the individual non-diagonal entries of $\mat{A}$.
		\begin{align*}
			\betam & =\frac{\am }{h^2_y},~~\betaf=\frac{\af }{h^2_y},~~
			\beta_I^F = \frac{\kappaf}{\kappaf+\kappam}\betam , \quad 
			\betaf_I =  \frac{\kappam}{\kappaf+\kappam}\betaf, \\
			\alpham &=\frac{\am }{h^2_x},~~\alpha^{F+}=\frac{\af }{h^2_x},~~\alpha^{F-}=\frac{\af }{h^2_x}+\frac{\vconst }{h_x}.		
		\end{align*}
		\begin{lemma}\label{lem_DD}
			The matrix $\mat{A}=\mat{A}(t)$ is weakly diagonal dominant for all $t\in[0,T]$.
		\end{lemma}
		\begin{proof}
			Inspecting the quantities $J_{i_l}(\mat{A})$ and  Table \ref{table_Rowsum} it can be seen that it holds  $J_{i_l}(\mat{A})\ge 0$, hence by Definition \ref{def:diag:dom} the matrix is diagonal dominant. 
		\end{proof}
		Note that $\mat{A}$ is weakly but not strictly diagonal dominant since not all of its rows are strictly diagonal dominant. 
		
		\begin{lemma}\label{lem_eigvalueA}
			The Gershgorin circles of the matrix $\mat{A}=\mat{A}(t)$ are subsets of $\mathbb{C}_-\cup\{0\}$  for all $t\in[0,T]$.
			Here, $\mathbb{C}_-$ denotes the set of complex numbers with negative real part.
		\end{lemma}
		\begin{proof}
			Let us examine the Gershgorin's circles of ${\mat{A}}$ for the  14 different representative rows  denoted by $D_{i_l}=D_{i_l}(C_{i_l},R_{i_l})$ with centres $C_{i_l}=\mat{A}_{i_li_l}$ and the radii $R_{i_l}(\mat{A}),~ {l}=1,\ldots, 14$, 
			given in Table \ref{table_Rowsum}.
			Since all entries of $\mat{A}$ are real, the centres $C_{i_l}=\mat{A}_{{i_l}{i_l}}<0$ of the discs are on the negative real axis. 
			Lemma \ref{lem_DD} shows that $\mat{A}$ is diagonal dominant, i.e., $J_{i_l}(\mat{A})= |C_{i_l}|- R_{i_l}(\mat{A}) \ge 0$. Hence, the radii $R_{i_l}(\mat{A})$ of the Gershgorin circles  never exceed $|C_{i_l}|$ and it holds $D_{i_l}\subset \mathbb{C}_-\cup\{0\}$.
			\qed
		\end{proof}
		
		\begin{lemma}\label{lem_SC}
			The matrix $\mat{A}=\mat{A}(t)$ is strongly connected for all $t\in[0,T]$.
		\end{lemma}
		\begin{proof}
			Let $(p,q)$ be a  pair of distinct integers with $p, q \in \{1, \ldots,n\}$. Then we can choose the sequence of distinct integers $k_1,k_2, \ldots,k_m,$ such that  $m=|p-q|+1$ and  $k_j=p+j-1$~(for $p<q$) and $k_j=p-j+1$~(for $~p>q$). It holds  $	{A}_{k_jk_{j+1}}\neq 0$ since these entries are located on the upper and lower subdiagonal of $\mat{A}$ for which we have
			\begin{align*}
				{A}_{k_jk_{j+1}}=\begin{cases}
					\betaFm, \qquad \text{for} \quad (k_j,k_{j+1}) \in \Nfm \setminus \Ninter_N,\\
					\betaFm_I, \qquad \text{for} \quad (k_j,k_{j+1}) \in \Ninter_N,
				\end{cases}
			\end{align*}
			where $\Nfm $ is the set of grid points in the fluid and medium $\Df \cup \Dm$ and $\Ninter_N$ the set of neighboring grid points to the interface. Since $\betaFm$ given in \eqref{coeff_schem}, \eqref{coeff_schem_medium} and $\betaFm_I $ given in \eqref{medium_I}, \eqref{fluid_I} are positive, we have $\mat{A}_{k_jk_{j+1}} \neq 0,~j=1,2,\ldots,m$. Thus, the matrix $\mat{A}$ is strongly connected. \qed
		\end{proof}
		\begin{lemma}\label{lem_nonsingular}
			The matrix $\mat{A}=\mat{A}(t)$ is non-singular for all $t\in[0,T]$.	
		\end{lemma}
		\begin{proof}
			From Lemma \ref{lem_DD} and \ref{lem_SC} it is known that $\mat{A}(t)$ is weakly diagonal dominant and strongly connected for all $t\in[0,T]$. Table \ref{table_Rowsum} shows that  there exist strictly diagonal dominant rows. Hence,  Better's Corollary (see Lemma \ref{better Corr})  implies that   $\mat{A}(t)$ is nonsingular.	
			\qed	
		\end{proof}
		
		\begin{lemma}\label{lem_max_norm_A}
			For  the maximum norm of the matrix $\mat{A}=\mat{A}(t)$ it holds 
			\begin{align*}
				\max_{t\in[0,T]} \|\mat{A}(t)\|_\infty =\max \big\{\big\|\mat{A}^{P}\big\|_{\infty},~\big\|\mat{A}^{N}\big\|_{\infty}\big\}
				\le 4 \max\{\af ,\am \}\Big(\frac{1}{h_x^2}+\frac{1}{h_y^2}\Big)+\frac{2\vconst }{h_x}.
			\end{align*}
		\end{lemma}
		\begin{proof}
			$\mat{A}(t)$ is  piecewise constant taking only the two values $\mat{A}^P$ and $\mat{A}^N$. From the last column of Table \ref{table_Rowsum} showing the 14 different row sums $S_{i}(\mat{A}^{P/N})$ of the two matrices  it can be easily seen that  $\|\mat{A}^{P/N}\|_\infty \le \max\{S_{i_6}(\mat{A}^{P/N}),S_{i_7}(\mat{A}^{P/N})\}$  yielding the estimate in the lemma.
			\qed	
		\end{proof}

		\begin{landscape}
			\begin{table}[p]
				\begin{center}
					\[\begin{array}{|c|c|c|c|c|c|}      
						\hline
						& \multicolumn{2}{|c|}{\text{Gershgorin circles: centres}}& \text{Gershgorin circles: radii}& \text{  Differences} & \text{  Row sums }\\
						l& \multicolumn{2}{|c|}{C_{i_l}=\mat{A}_{i_li_l} \quad \text{(diagonal entries)}} &  ~R_{i_l}(\mat{A})=\displaystyle\sum_{j=1,j \neq i_l}^{n} |\mat{A}_{i_l,j}|&J_{i_l}(\mat{A})=|\mat{A}_{i_li_l}|-\displaystyle\sum_{j=1,j \neq i_l}^{n} |\mat{A}_{i_l,j}| &  \quad S_{i_l}(\mat{A})=\displaystyle\sum_{j=1}^{n} |\mat{A}_{i_l,j}| \\\hline						
						1 &\gamam_{DB}& -\am  \big(\frac{1}{h^2_x}+\frac{1}{h^2_y}\big) -\big(\frac{\heattransfer h_y}{\kappam+\heattransfer h_y}\big)\frac{\am }{h^2_y} & \am  \big(\frac{1}{h^2_x}+\frac{1}{h^2_y}\big) &\big(\frac{\heattransfer h_y}{\kappam+\heattransfer h_y}\big)\frac{\am }{h^2_y}&
						2\am  \big(\frac{1}{h^2_x}+\frac{1}{h^2_y}\big) +\big(\frac{\heattransfer h_y}{\kappam+\heattransfer h_y}\big)\frac{\am }{h^2_y}\\[1.5ex]
						2&\gamam_{UB}&-\am \big(\frac{1}{h^2_x}+\frac{1}{h^2_y}\big)&\am \big(\frac{1}{h^2_x}+\frac{1}{h^2_y}\big)&0&2\am \big(\frac{1}{h^2_x}+\frac{1}{h^2_y}\big)\\[1.5ex]
						3 &\gamam_D& -\am  \big(\frac{2}{h^2_x}+\frac{1}{h^2_y}\big) -\big(\frac{\heattransfer h_y}{\kappam+\heattransfer h_y}\big)\frac{\am }{h^2_y}& \am  \big(\frac{2}{h^2_x}+\frac{1}{h^2_y}\big)&
						\big(\frac{\heattransfer h_y}{\kappam+\heattransfer h_y}\big)\frac{\am }{h^2_y}&
						2\am  \big(\frac{2}{h^2_x}+\frac{1}{h^2_y}\big) +\big(\frac{\heattransfer h_y}{\kappam+\heattransfer h_y}\big)\frac{\am }{h^2_y}\\[1.5ex]
						4&\gamam_U&-\am \big(\frac{2}{h^2_x}+\frac{1}{h^2_y}\big)& \am \big(\frac{2}{h^2_x}+\frac{1}{h^2_y}\big)&0&2\am \big(\frac{2}{h^2_x}+\frac{1}{h^2_y}\big)\\[1.5ex]	
						5 &\gamam_B& -\am \big(\frac{1}{h^2_x}+\frac{2}{h^2_y}\big)&\am \big(\frac{1}{h^2_x}+\frac{2
						}{h^2_y}\big)&0&2\am \big(\frac{1}{h^2_x}+\frac{2}{h^2_y}\big) \\[1.5ex]	\hline
						6 &\gamam & -2\am \big(\frac{1}{h^2_x}+\frac{1}{h^2_y}\big)& 2\am \big(\frac{1}{h^2_x}+\frac{1}{h^2_y}\big)&0&4\am \big(\frac{1}{h^2_x}+\frac{1}{h^2_y}\big)\\[1.5ex]	
						7 &\gamaf & -2\af \big(\frac{1}{h^2_x}+\frac{1}{h^2_y}\big)-\frac{v_0(t) }{h_x}&2\af \big(\frac{1}{h^2_x}+\frac{1}{h^2_y}\big)+\frac{v_0(t) }{h_x}&0&4\af \big(\frac{1}{h^2_x}+\frac{1}{h^2_y}\big)+\frac{2v_0(t) }{h_x}\\[1.5ex]	
						8 &\gamaf_L& 
						\begin{cases}
							-2\af \big(\frac{1}{h^2_x}+\frac{1}{h^2_y}\big)-\frac{\vconst }{h_x}, & \mat{A}=\mat{A}^P\\
							-\af \big(\frac{1}{h^2_x}+\frac{2}{h^2_y}\big),& \mat{A}=\mat{A}^N\\
						\end{cases}& \af \big(\frac{1}{h^2_x}+\frac{2}{h^2_y}\big)&
						\begin{cases}
							\frac{\af }{h^2_x}+\frac{\vconst }{h_x}, &\mat{A}=\mat{A}^P\\
							0,& \mat{A}=\mat{A}^N\\
						\end{cases}&
						\begin{cases}
							\phantom{2}\af \big(\frac{3}{h^2_x}+\frac{4}{h^2_y}\big)+\frac{\vconst }{h_x},& \mat{A}=\mat{A}^P\\
							2\af \big(\frac{1}{h^2_x}+\frac{2}{h^2_y}\big), &\mat{A}=\mat{A}^N\\
						\end{cases}\\[1.5ex]	
						9 &\gamaf_R&  
						-\af \big(\frac{1}{h^2_x}+\frac{2}{h^2_y}\big)-\frac{v_0(t) }{h_x}&\af \big(\frac{1}{h^2_x}+\frac{2}{h^2_y}\big)+\frac{v_0(t) }{h_x}&0&2\af \big(\frac{1}{h^2_x}+\frac{2}{h^2_y}\big)+\frac{2v_0(t) }{h_x}\\[1.5ex]	\hline
						10 &\gamam _{I} &  -\am \big(\frac{2}{h^2_x}+\frac{1}{h^2_y}\big) -\big(\frac{\kappaf}{\kappam+\kappaf}\big)\frac{\am }{h^2_y}& \frac{2\am }{h^2_x} +\big(1+\frac{\kappaf}{\kappam+\kappaf}\big)\frac{\am }{h^2_y}& 0& 2\am \big(\frac{2}{h^2_x}+\frac{1}{h^2_y}\big)+\big(\frac{2\kappaf}{\kappam+\kappaf}\big)\frac{\am }{h^2_y}\\[1.5ex]	
						11 &\gamam _{IB} & -\am \big(\frac{1}{h^2_x}+\frac{1}{h^2_y}\big) -\big(\frac{\kappaf}{\kappam+\kappaf}\big)\frac{\am }{h^2_y} & \frac{\am }{h^2_x} +\big(1+\frac{\kappaf}{\kappam+\kappaf}\big)\frac{\am }{h^2_y}& 0& 2\am \big(\frac{1}{h^2_x}+\frac{1}{h^2_y}\big)+\big(\frac{2\kappaf}{\kappam+\kappaf}\big)\frac{\am }{h^2_y}\\[1.5ex]	
						12&\gamaf _{I} & -\af \big(\frac{2}{h^2_x}+\frac{1}{h^2_y}\big) -\big(\frac{\kappam }{\kappam+\kappaf}\big)\frac{\af }{h^2_y} -\frac{v_0(t)}{h_x}& 
						\frac{2\af }{h^2_x} +\big(1+\frac{\kappam }{\kappam+\kappaf}\big)\frac{\af }{h^2_y} +\frac{v_0(t) }{h_x}&0&2\af \big(\frac{2}{h^2_x}+\frac{1}{h^2_y}\big)+\big(\frac{2\kappam }{\kappam+\kappaf}\big)\frac{\af }{h^2_y}+\frac{2v_0(t) }{h_x}\\[1.5ex]
						13&\gamaf _{IL} & 
						\begin{cases}
							-\af \big(\frac{2}{h^2_x}+\frac{1}{h^2_y}\big) -\big(\frac{\kappam }{\kappam+\kappaf}\big)\frac{\af }{h^2_y}-\frac{\vconst }{h_x},& \mat{A}=\mat{A}^P\\[1ex]
							-\af \big(\frac{1}{h^2_x}+\frac{1}{h^2_y}\big) -\big(\frac{\kappam }{\kappam+\kappaf}\big)\frac{\af }{h^2_y},& \mat{A}=\mat{A}^N\\
						\end{cases}& \frac{\af }{h^2_x} +\big(1+\frac{\kappam }{\kappam+\kappaf}\big)\frac{\af }{h^2_y}&
						\begin{cases}
							\frac{\af }{h^2_x}+\frac{\vconst }{h_x},&\mat{A}=\mat{A}^P\\
							0,&\mat{A}=\mat{A}^N\\
						\end{cases}&
						\begin{cases}
							\phantom{2}\af \big(\frac{3}{h^2_x}+\frac{2}{h^2_y}\big) +\big(\frac{2\kappam }{\kappam+\kappaf}\big)\frac{\af }{h^2_y} +\frac{\vconst }{h_x},&\mat{A}=\mat{A}^P\\[1ex]
							2\af \big(\frac{1}{h^2_x}+\frac{1}{h^2_y}\big) +\big(\frac{2\kappam }{\kappam+\kappaf}\big)\frac{\af }{h^2_y},& \mat{A}=\mat{A}^N\\
						\end{cases}\\[3.5ex]	
						14 &	\gamaf _{IR} &
						-\af \big(\frac{1}{h^2_x}+\frac{1}{h^2_y}\big) -\big(\frac{\kappam }{\kappam+\kappaf}\big)\frac{\af }{h^2_y} 
						-\frac{v_0(t) }{h_x}&
						\frac{\af }{h^2_x} +\big(1+\frac{\kappam }{\kappam+\kappaf}\big)\frac{\af }{h^2_y} +\frac{v_0(t) }{h_x}&0&2\af \big(\frac{1}{h^2_x}+\frac{1}{h^2_y}\big)+\big(\frac{2\kappam }{\kappam+\kappaf}\big)\frac{\af }{h^2_y}+\frac{2v_0(t) }{h_x}\\	
						\hline
					\end{array}\]
					\mycaption{Diagonal entries $\mat{A}_{i_li_l}$ (centres  of Gershgorin circles), radii of Gershgorin circles $R_{i_l}$, differences  $J_{i_l}$  and row sums  $S_{i_l}$ of matrices $\mat{A}^{P}$ and $\mat{A}^{N}$, $l=1,\ldots,14$}
					\label{table_Rowsum}
				\end{center}
			\end{table} 
		\end{landscape}

		\section{Proof of Lemma \ref{lem:bounds_GHF}}
		\label{append_stability_scheme}
		
		\begin{proof}
			\paragraph{First assertion}
			Table~\ref{table_Rowsum} shows that the diagonal entries of the matrices $\mat{A}^k$, $k=1,\ldots,N_\tau$ are all negative. Thus, we have for all $i=1,\ldots,n$
			\begin{align*}
				J_i(\mat{G}^k)=|\mat{G}^{k}_{ii}|- \sum_{j=1, j\neq i} |\mat{G}^{k}_{ij}|=1+\tau \theta (|\mat{A}^k_{ii}|- \sum_{j=1, j\neq i} |\mat{A}^k_{ij}|)=1+\tau \theta J_i(\mat{A}^k) \ge 1, 
			\end{align*}
			since by Lemma \ref{lem_DD} the matrices $\mat{A}^k$ are diagonal dominant and it holds  $J_i(\mat{A}^k) \geq 0$.
			Therefore, the matrices $\mat{G}^{k}=\mathds{I}_{n}-\tau \theta \mat{A}^{k}, ~~k=1,\ldots,N_\tau$ are strictly diagonal dominant. Lemma~\ref{Inverse} implies that $\mat{G}^{k}$ is invertible  and $\|(\mat{G}^{k})^{-1}\|_{\infty} \leq 1/{J(\mat{G}^{k})} \le 1$. For $\theta=0$ it holds $\mat{G}^{k}=\mathds{I}_{n}$, hence $\|(\mat{G}^{k})^{-1}\|_{\infty}=\|\mathds{I}_{n}\|_{\infty}=1$ and the above inequality holds with equality.
			\\[1ex]
			\paragraph{Second assertion} 
			We recall the definition of $\mat{H}^{k} $ given in \eqref{Matrix_form2} which reads as $\mat{H}^{k} =\mathds{I}_{n}+\tau(1-\theta)\mat{A}^k$.
			For $\theta=1$, we have  $\mat{H}^{k}=\mathds{I}_n$, thus  for all $\tau >0 $ it holds $\big\|\mat{H}^{k}\big\|_{\infty}=1$ which proves the claim for $\theta=1$.
			
			Now, let $\theta\in[0,1)]$. We recall that $\mat{A}^k=\mat{A}(k\tau)$ takes only the values $\mat{A}^P$ and $\mat{A}^N$. Thus it is sufficient to show that the claim holds for  $\mat{H}^P$ and $\mat{H}^N$  where  $\mat{H}^{P/N} =\mathds{I}_{n}+\tau(1-\theta)\mat{A}^{P/N}$.
			It holds
			\begin{align*}
				\big\|\mat{H}^{P}\big\|_{\infty}=\displaystyle\max_{1\leq i\leq n}\big\{S_i(\mat{H}^{P})\big\},~~\text{with}~~	S_i(\mat{H}^{P})&=|1+\tau(1-\theta)\mat{A}^{P}_{ii}|+\tau(1-\theta)\displaystyle \sum_{j=1, j\neq i}^{n}|\mat{A}^{P}_{ij}|.
			\end{align*}
			Using the fact that all diagonal entries of the matrix $\mat{A}^P$ are negative, we have for $\tau_i^P=\frac{1}{(1-\theta) |\mat{A}^{P}_{ii}|}$, $i=1,\ldots,n$,  
			\begin{align*}
				|1+\tau(1-\theta)\mat{A}^{P}_{ii}|= \begin{cases} 
					1-\tau(1-\theta)|\mat{A}^{P}_{ii}|,&\text{for}~ \tau \leq\tau_i^P\\
					\tau(1-\theta)|\mat{A}^{P}_{ii}|-1,& \text{for}~ \tau >\tau_i^P.
				\end{cases}
			\end{align*}
			This implies that for $i=1,\ldots,n$, we have
			\begin{align*}
				S_i(\mat{H}^{P})=\begin{cases}
					1-\tau(1-\theta)\Big[|\mat{A}^{P}_{ii}|-\displaystyle \sum_{j=1, j\neq i}^{n}|\mat{A}^{P}_{ ij}|\Big]=1-\tau(1-\theta)R_i(\mat{A}^{P}),& \text{for}~ \tau \leq\tau_i^P,\\
					\tau(1-\theta)\Big[|\mat{A}^{P}_{ii}|+\displaystyle \sum_{j=1, j\neq i}^{n}|\mat{A}^{P}_{ ij}|\Big]-1=-1+ \tau(1-\theta)S_i(\mat{A}^{P}),& \text{for}~ \tau >\tau_i^p.
				\end{cases}
			\end{align*}
			Since $\mat{A}^P$ is weakly diagonal dominant, we distinguish the two cases $J_i(\mat{A}^{P})>0$ and $J_i(\mat{A}^{P})=0$.\\
			For $J_i(\mat{A}^{P})>0$, the sum $S_i(\mat{H}^{P})$ is strictly decreasing in $\tau$ on $[0,\tau_i^P]$ and strictly increasing in $\tau$ on $(\tau_i^P,+\infty)$ and it holds
			\begin{align*}
				S_i(\mat{H}^{P}) \leq 1 ~~\text{for}~~\tau  \leq \overline{\tau}_i^{P}:=\frac{2}{(1-\theta)S_i(\mat{A}^{P})}~~ \text{and}~~S_i(\mat{H}^{P})>1  ~~\text{for}~~\tau > \overline{\tau}_i^{P}.
			\end{align*}
			For $J_i(\mat{A}^{P})=0$, we have $S_i(\mat{A}^{P})=2|\mat{A}^{P}_{ii}|$. It holds $S_i(\mat{H}^{P})=1$ for $\tau \in [0,\overline{\tau}_i^P]$ while $S_i(\mat{H}^{P})$ is  strictly increasing in $\tau$ on $(\overline{\tau}_i^P,+\infty)$, hence $S_i(\mat{H}^{P})>1  ~~\text{for}~~\tau > \overline{\tau}_i^{P}$.\\
			Summarizing we obtain 
			\begin{align*}
				&\big\|\mat{H}^{P}\big\|_{\infty}=\displaystyle\max_{1\leq i\leq n}S_i(\mat{H}^{P})=1   ~~\text{for}~~\tau \leq \overline{\tau}^{P}=\displaystyle\min_{1\leq i\leq n} \overline{\tau}^{P}_i=\frac{2}{(1-\theta)\displaystyle \max_{1\leq i\leq n}S_i(\mat{A}^{P})}
				=\frac{2}{(1-\theta)\|\mat{A}^{P}\|_{\infty}}, 		
			\end{align*}
			and $\|\mat{H}^{P}\big\|_{\infty} >1$ for $\tau > \overline{\tau}^{P}$.
			For $\mat{A}=\mat{A}^{N}$ the proof is analogous.
			Thus, we have 
			\begin{align*}
				\big\|\mat{H}^{k}\big\|_{\infty}\le 1 ~~~\text{for}~~\tau \le \min\{\overline{\tau}^{P},\overline{\tau}^{N}\} = \frac{2}{(1-\theta) 
					\max \big\{\big\|\mat{A}^{P}\big\|_{\infty},~\big\|\mat{A}^{N}\big\|_{\infty}\big\}}.		
			\end{align*}
			Finally, Lemma \ref{lem_max_norm_A} shows that $\max \big\{\big\|\mat{A}^{P}\big\|_{\infty},~\big\|\mat{A}^{N}\big\|_{\infty}\big\}=
			4 \max\{\af ,\am \}\Big(\frac{1}{h_x^2}+\frac{1}{h_y^2}\Big)+\frac{2\vconst }{h_x}=2\eta
			$ which proves the claim.
			\\
			\paragraph{Third assertion} From the definition of ${F}^{k} $ given in \eqref{Matrix_form2} it follows that  for $k=0,\ldots,N_\tau-1$
			\begin{align*}
				\big\|{F}^{k} \big\|_{\infty}& = \big\|\theta \mat{B}^{k+1}g^{k+1}+(1-\theta)\mat{B}^{k}g^{k} \big\|_{\infty}
				\leq \theta \big\|\mat{B}^{k+1}\big\|_{\infty} \big\|g^{k+1}\big\|_{\infty}+(1-\theta) \big\|\mat{B}^{k}\big\|_{\infty} \big\|g^{k}\big\|_{\infty}	\\
				&\le  (\theta +1-\theta) C_B \max_{j=k,k+1}\big\|g^{j}\big\|_{\infty} \le C_B\displaystyle\max_{0\leq j\leq k+1}\big\|g^{j}\big\|_{\infty}.
			\end{align*} 
			where we have used that $\mat{B}^k=\mat{B}(k\tau)$ takes only the values $\mat{B}^P$ and $\mat{B}^N$.
			\qed
		\end{proof}

	\end{appendix}	
	
	\begin{acknowledgements}
		The authors thank  Thomas Apel (Universität der Bundeswehr München), Martin Bähr, Michael Breuss, Carsten Hartmann, Gerd Wachsmuth (BTU Cottbus--Senftenberg), Andreas Witzig (ZHAW Winterhur), Karsten Hartig (Energie-Concept Chemnitz), Dietmar Deunert, Regina Christ  (eZeit Ingenieure Berlin) for valuable discussions	that improved this paper.\\
		P.H.~Takam gratefully acknowledges the  support by the German Academic Exchange Service \linebreak[4] (DAAD) within the project ``PeStO – Perspectives in Stochastic Optimization and Applications''. \\
		R.~Wunderlich gratefully acknowledges the  support by the Federal Ministry of Education and Research (BMBF) within the project ``05M2022 - MONES: Mathematische Methoden für die Optimierung von	Nahwärmenetzen und Erdwärmespeichern''.\\
		The work of O. Menoukeu Pamen was supported with funding provided by the Alexander von Humboldt Foundation, under the programme financed by the German Federal Ministry of Education and Research entitled German Research Chair No 01DG15010.
	\end{acknowledgements}
	
	\bibliographystyle{acm}

\end{document}